\documentclass[journal]{IEEEtran}
\usepackage{amsmath, amssymb, amsthm}
\usepackage{graphicx} 
\usepackage{cleveref}
\usepackage{tabularray}

\usepackage{tikz}
\usetikzlibrary{fit,backgrounds,positioning,calc}
\usetikzlibrary{arrows.meta}
\tikzset{>=latex}

\graphicspath{{./Figs/}}

\newcommand{\E}{\mathbb{E}}

\newcommand{\G}{G}
\newcommand{\liealgebra}{\mathfrak{g}}

\newcommand{\pair}[1]{\ensuremath{\left\langle #1 \right\rangle}}
\renewcommand{\Re}{\ensuremath{\mathbb{R}}}

\newcommand{\deriv}[2]{\ensuremath{\frac{\partial #1}{\partial #2}}}

\newcommand{\braces}[1]{\ensuremath{\left\{ #1 \right\}}}
\newcommand{\parenth}[1]{\ensuremath{\left( #1 \right)}}
\renewcommand{\div}{\mathrm{div}}
\newcommand{\grad}{\mathrm{grad}}
\newcommand{\Sph}{\mathsf{S}}
\newcommand{\SO}[1]{\mathrm{SO}(#1)}
\newcommand{\so}[1]{\mathfrak{so}(#1)}
\newcommand{\Hessmatrix}{\overline{\mathrm{H}}}
\newcommand{\g}{\ensuremath{\mathsf{g}}}
\newcommand{\Ad}{\ensuremath{\mathrm{Ad}}}
\newcommand{\ad}{\ensuremath{\mathrm{ad}}}
\newcommand{\conn}{\ensuremath{\tilde\nabla}}

\newcommand{\Aff}{\ensuremath{\mathsf{Aff}}}

\newcommand{\Hyper}{\ensuremath{\mathbb{H}}}

\newtheorem{definition}{Definition}
\newtheorem{theorem}{Theorem}
\newtheorem{corollary}{Corollary}
\newtheorem{remark}{Remark}

\newtheorem{prop}{Proposition}

\title{Geometric Interpretation of Brownian Motion on Riemannian Manifolds}

\author{Taeyoung Lee and Gregory S. Chirikjian
    \thanks{Taeyoung Lee, Mechanical and Aerospace Engineering, George Washington University, Washington, DC 20052, {\tt \{tylee\}@gwu.edu}}%
    \thanks{Gregory S. Chirikjian, Willis F. Harrington Professor of Mechanical Engineering, University of Delaware, Newark DE 19716 {\tt gchirik@udel.edu}}
    \thanks{\textsuperscript{\footnotesize\ensuremath{*}}This research has been supported in part by AFOSR MURI FA9550-23-1-0400, and ONR N00014-23-1-2850.}
}

\begin{document}
	
\maketitle

\begin{abstract}
    This paper presents a unified geometric framework for Brownian motion on manifolds, encompassing intrinsic Riemannian manifolds, embedded submanifolds, and Lie groups.  
    The approach constructs the stochastic differential equation by injecting noise along each axis of an orthonormal frame and designing the drift term so that the resulting generator coincides with the Laplace--Beltrami operator.  
    Both Stratonovich and It\^{o} formulations are derived explicitly, revealing the geometric origin of curvature-induced drift.  
    The drift is shown to correspond to the covariant derivatives of the frame fields for intrinsic manifolds, the mean curvature vector for embedded manifolds, and the adjoint-trace term for Lie groups, which vanishes for unimodular cases.  
    The proposed formulation provides a geometrically transparent and mathematically consistent foundation for diffusion processes on nonlinear configuration spaces.
\end{abstract}

\section{Introduction}

Brownian motion is one of the most fundamental stochastic processes in probability theory and mathematical physics.  
Originally introduced to describe the random motion of particles suspended in a fluid~\cite{chirikjian2000engineering}, it has since become a cornerstone of modern stochastic analysis and a canonical model for diffusion phenomena.  
In Euclidean space, Brownian motion admits several equivalent characterizations.  
From a probabilistic perspective, it is defined as the \emph{Wiener process}---a continuous stochastic process with independent, stationary Gaussian increments and almost surely continuous sample paths~\cite{oksendal2003stochastic}.  
Alternatively, from an analytical viewpoint, Brownian motion can be characterized as the diffusion process whose infinitesimal generator equals one-half of the Laplacian operator~\cite{oksendal2003stochastic}.  
This interpretation establishes a direct connection between stochastic differential equations and partial differential equations, as the Kolmogorov forward and backward equations associated with Brownian motion correspond to the classical heat equation.

This characterization via the Laplacian provides a natural pathway for generalization beyond Euclidean space.  
On a Riemannian manifold, the Laplacian is replaced by the \emph{Laplace--Beltrami operator}~\cite{marsden2013introduction}, denoted by $\Delta$, which encodes the intrinsic geometry of the manifold.  
Accordingly, Brownian motion on a Riemannian manifold is defined as the diffusion process whose generator equals $\tfrac{1}{2}\Delta$~\cite{elworthy1982stochastic,ikeda1989stochastic,hsu2002stochastic}.  
This geometric generalization forms a central object in stochastic geometry and geometric analysis, providing a probabilistic interpretation of heat diffusion on curved spaces and a stochastic counterpart to geodesic flow.  
It also underlies many modern models in mechanics, control, and robotics, where system states evolve on nonlinear configuration spaces.  

Understanding how Brownian motion interacts with curvature and topological structure thus lies at the intersection of differential geometry and stochastic analysis, motivating the geometric formulations developed in this paper.  
Several formulations of Brownian motion on manifolds have been developed, each reflecting a different geometric perspective.  

\paragraph*{Intrinsic Formulation}
The \emph{intrinsic} or \emph{orthonormal frame-bundle} approach, pioneered by Eells, Elworthy, and Malliavin~\cite{elworthy1982stochastic,ikeda1989stochastic,hsu2002stochastic}, constructs Brownian motion on a Riemannian manifold as the projection of a diffusion process evolving on the manifold’s orthonormal frame bundle.  
Specifically, at each point of the manifold, this bundle collects all possible orthonormal bases of the tangent space, and a Euclidean Brownian motion is lifted to the bundle, where it evolves in a manner consistent with parallel transport along random paths.  
Projecting the resulting process back to the manifold yields a Brownian motion whose generator is $\tfrac{1}{2}\Delta$.  

This construction provides an elegant and fully intrinsic definition of stochastic processes on manifolds and later became the foundation for defining general stochastic differential equations on manifolds.  
However, its generality comes at the cost of complexity.  
Because the process evolves on the high-dimensional frame bundle rather than the manifold itself, it requires simultaneously tracking the base point and a moving orthonormal frame.  
In practice, this entails expressing connection forms and curvature terms in local coordinates, making analytical derivations and numerical implementations cumbersome.  
Moreover, the resulting drift term is typically implicit and lacks a transparent geometric interpretation in terms of the curvature of the manifold.

\paragraph*{Extrinsic Formulation}
A complementary \textit{extrinsic} strategy, particularly suited for embedded manifolds, is to formulate the stochastic differential equation in the ambient Euclidean space~\cite{elworthy1982stochastic,hsu2002stochastic}.  
This approach involves projecting the Euclidean Brownian motion onto the tangent bundle of the embedded manifold and introducing a drift term to ensure that the resulting process remains on the manifold.  
In particular, Lewis~\cite{lewis1986brownian} provides an explicit It\^{o} formulation for Brownian motion on a submanifold of Euclidean space given by a level set of a smooth function. 

However, most existing results, including~\cite{lewis1986brownian}, are restricted to the ambient Euclidean setting and do not naturally extend to intrinsic manifolds or to Lie groups endowed with non-Euclidean metrics.  
Furthermore, the derivation is typically confined to the It\^{o} form of the stochastic differential equation, while the corresponding Stratonovich formulation---which is geometrically more natural and directly compatible with the manifold’s connection---is usually omitted.  
As a result, the relationship between the It\^{o} and Stratonovich representations, and how the curvature-induced drift arises from their conversion, has not been systematically clarified.  
Finally, the geometric meaning of this drift has been established only in the Euclidean embedding sense, without unifying it with the intrinsic Laplace--Beltrami framework or with invariant constructions on Lie groups.  

\paragraph*{Lie Group Formulation}
For manifolds possessing an inherent group structure, such as matrix Lie groups, their algebraic properties can be exploited to define globally consistent stochastic differential equations intrinsically.  
Specifically, in his seminal work~\cite{ito1950brownian}, It\^{o} provided the first mathematical definition and construction of Brownian motion on a Lie group by characterizing its generator as a second-order operator with constant coefficients in the Lie algebra, and then using this algebraic form to build a corresponding stochastic differential equation in local coordinates.  
Later, by employing left-invariant vector fields and the Maurer--Cartan form, the stochastic process can be expressed directly in the Lie algebra, bypassing the need for either local charts or external embeddings~\cite{chirikjian2011stochastic}.  
These ideas have been applied to specific Lie groups in~\cite{piggott2016geometric,solo2010attitude}.  

Despite their elegance, existing Lie-group formulations remain largely algebraic and lack an explicit geometric interpretation of the drift term.  
It\^{o}’s original construction predates Stratonovich calculus and expresses Brownian motion through coordinate-dependent It\^{o} equations, without linking the resulting drift to the curvature of the manifold.  
Later extensions using left-invariant vector fields, including~\cite{chirikjian2011stochastic,piggott2016geometric,solo2010attitude}, focus on specific matrix groups but do not explicitly connect these formulations to the general theory of Brownian motion on Riemannian manifolds.  
In particular, the crucial simplification that arises for the broad class of unimodular Lie groups—where the Stratonovich drift vanishes completely—and the contrasting treatment required for non-unimodular cases have not been presented within a unified geometric framework.

In short, while Brownian motion on a manifold has been investigated extensively, each study focuses on a specific formulation, and there is a lack of geometric interpretation that connects these results in a unified and comprehensive way.  

\paragraph*{Objectives}
The objective of this paper is to bridge these gaps by providing a unified and geometrically intuitive development of stochastic differential equations for Brownian motion, without resorting to the frame-bundle formulation.  
First, we systematically derive the intrinsic, extrinsic, and Lie-group formulations from the common axiomatic definition of Brownian motion as the process generated by the Laplace--Beltrami operator.  
Second, we provide explicit, self-contained proofs for the drift terms that arise in both the It\^{o} and Stratonovich forms for each framework, giving them clear geometric and algebraic interpretations.  
Third, we specialize these results to Lie groups, demonstrating how the geometric drift simplifies to a purely algebraic expression and we prove that this drift vanishes for the entire class of unimodular Lie groups.  
Finally, the theoretical results are validated and illustrated with a comprehensive suite of examples.

\paragraph*{Contributions}
Overall, this work establishes a geometrically transparent and mathematically consistent foundation for Brownian motion on manifolds.  
It bridges differential geometry and stochastic analysis by identifying the drift term as a fundamental geometric quantity while maintaining compatibility with both intrinsic and extrinsic representations.  
The resulting framework not only generalizes prior constructions but also provides explicit, coordinate-free stochastic differential equations suitable for analytical study and numerical implementation on curved configuration spaces.  

The main contributions are summarized as follows:
\begin{itemize}
    \item Developed a unified geometric framework to construct Brownian motion on intrinsic manifolds (\Cref{thm:Brownian_motion}), embedded submanifolds (\Cref{thm:Brownian_motion_embedded}), and Lie groups (\Cref{thm:Brownian_motion_G,thm:Brownian_motion_G_Emb}) in a systematic way.
    \item Derived the Stratonovich drift term explicitly as an intrinsic geometric quantity obtained by the covariant derivatives of the frame, which has not been reported in a direct coordinate-free form for an arbitrary manifold (\Cref{thm:Brownian_motion,thm:Brownian_motion_embedded}).
    \item Established the It\^{o} drift term vanishes for an intrinsic formulation (\Cref{thm:Brownian_motion}), and it corresponds to the mean curvature vector for an extrinsic formulation (\Cref{thm:Brownian_motion_embedded}). 
        While \cite{lewis1986brownian} derived the mean curvature vector in Euclidean space, this works extends it by rigorously proving its equivalence to the Stratonovich formulation.
    \item Generalized Brownian motion on a Riemannian manifold to Lie groups to show that the Stratonovich drift is expressed by the co-adjoint operator, linking algebraic and geometric structures explicitly (\Cref{thm:Brownian_motion_G,thm:Brownian_motion_G_Emb}). 
        Further, it is shown that this term vanishes for unimodular group, and it is simplified for non-unimodular group---result not previously established (\Cref{cor:Brownian_motion_G_unimodular,cor:Brownian_motion_G_unimodular_Emb}).
\end{itemize}

The organization of this paper is illustrated in \Cref{fig:overview}.

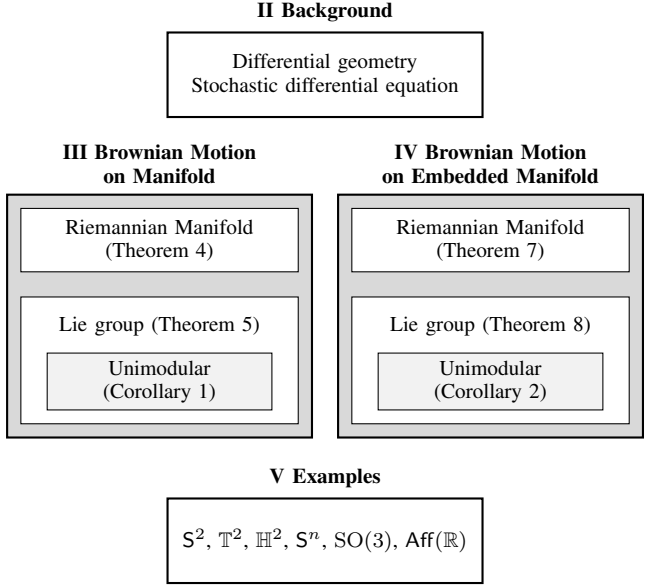
\begin{figure}
    \footnotesize
    \tikzstyle{block} = [draw, fill=gray!0, rectangle, minimum height=4em, minimum width=6em, text width=3.5cm, align=center]
    \tikzstyle{inner_block} = [draw, fill=gray!10, rectangle, minimum height=2.5em, minimum width=4em, text width=2.8cm, align=center]
    \tikzstyle{outer_block} = [thick, inner sep=5pt, fill=gray!30, draw] 
    \centerline{
        \begin{tikzpicture}
            \node[block, minimum height=3em] (M) at (0,0) {Riemannian Manifold (\Cref{thm:Brownian_motion})};
            \node[block, minimum height=6em] (G) at (0,-1.6) {};
            \node at ([yshift=0.45cm]G.center) {Lie group (\Cref{thm:Brownian_motion_G})};
            \node[inner_block, yshift=-1.0em] (UG) at (G.center) {Unimodular\\(\Cref{cor:Brownian_motion_G_unimodular})};
            \begin{scope}[on background layer]
                \node[outer_block, fit=(M) (G) (UG)] (RM) {};
            \end{scope}
            \node[block, minimum height=3em] (EM) at (4.4,0) {Riemannian Manifold (\Cref{thm:Brownian_motion_embedded})};
            \node[block, minimum height=6em] (EG) at (G-|EM) {};
            \node at ([yshift=0.45cm]EG.center) {Lie group (\Cref{thm:Brownian_motion_G_Emb})};
            \node[inner_block, yshift=-1.0em] (EUG) at (EG.center) {Unimodular\\(\Cref{cor:Brownian_motion_G_unimodular_Emb})};
            \begin{scope}[on background layer]
                \node[outer_block, fit=(EM) (EG) (EUG)] (ERM) {};
            \end{scope}
            \node[text width=3cm, align=center] at (RM.north) [above=1pt] {\textbf{\ref{sec:BMM} Brownian Motion\\ on Manifold}};
            \node[text width=3cm, align=center] at (ERM.north) [above=1pt] {\textbf{\ref{sec:BMEM} Brownian Motion\\ on Embedded Manifold}};
            
            \coordinate (mid) at ($ (RM.center)!0.5!(ERM.center) $);
            \node[block, thick, text width=4cm] (ex) at (mid |-0, -4) {$\Sph^2$, $\mathbb{T}^2$, $\mathbb{H}^2$, $\Sph^n$, $\SO3$, $\Aff(\Re)$ };
            \node at (ex.north) [above=1pt] {\textbf{\ref{sec:ex} Examples}};

            \node[block, thick, text width=4cm] (Background) at (mid |-0, 2.2) {Differential geometry\\Stochastic differential equation};
            \node at (Background.north) [above=1pt] {\textbf{\ref{sec:Background} Background}};
            
        \end{tikzpicture}
    }
    \caption{Overview of paper: \Cref{sec:Background} presents background materials. Then, \Cref{sec:BMM} develops Brownian motion on Riemannian manifolds and Lie groups along with a special case when the Lie group is unimodular. These are repeated for Riemannian manifolds and Lie groups embedded in the Euclidean space in \Cref{sec:BMEM}, which are followed by examples in \Cref{sec:ex}.}\label{fig:overview}
\end{figure}

\section{Background} \label{sec:Background}

We briefly review key concepts from differential geometry~\cite{lee2006riemannian,marsden2013introduction} and stochastic differential equation~\cite{ikeda1989stochastic,gliklikh2010global,chirikjian2009stochastic,chirikjian2011stochastic} that will be used throughout this paper.

\subsection{Riemannian Manifold}\label{sec:diff_geo}

Let $(M,\g)$ be a smooth $n$-dimensional Riemannian manifold. 
For a point $p \in M$, the \emph{tangent space} $T_p M$ is the vector space consisting of all possible tangent vectors at the point. 
The \textit{tangent bundle} $TM$ is the disjoint union of the tangent spaces at all points of the manifold.
A \emph{vector field} $X:M\rightarrow TM $ is a smooth map that assigns to each $p\in M$ a tangent vector $X(p) \in T_p M$.
The set of all vector fields on a manifold $M$ is denoted by $\mathfrak{X}(M)$.

Let $C^\infty(M)$ be the set of smooth real-valued functions on $M$.
For $f\in C^\infty (M)$ and $X\in\mathfrak{X}(M)$, the \textit{Lie derivative of $f$ along $X$} is the directional derivative given by
\begin{align}
    \mathcal{L}_X f = X[f] \triangleq df(X) \in C^\infty(M),
\end{align}
where $df$ is the exterior derivative, or the differential of $f$.  

For $X,Y\in\mathfrak{X}(M)$, the \textit{Jacobi--Lie} bracket of $X$ and $Y$, denoted by $[\cdot, \cdot]:\mathfrak{X}(M)\times\mathfrak{X}(M)\rightarrow\mathfrak{X}(M)$ is the unique vector field defined by
\begin{align}
    ([X, Y]) [f] = X[Y[f]] - Y[X[f]]. \label{eqn:Jac-Lie}
\end{align}
This coincides with the \textit{Lie derivative of $Y$ along $X$}, denoted by $\mathcal{L}_X Y$.
The Jacobi--Lie bracket is anticommutative, and it satisfies the following Jacobi identity:
\begin{align}
    [[X,Y],Z] + [[Y,Z],X] + [[Z, X],Y] = 0, \label{eqn:Jacobi}
\end{align}
for any $X,Y,Z\in\mathfrak{X}(M)$. 

An \textit{affine connection}, $\conn:\mathfrak{X}(M)\times\mathfrak{X}(M)\rightarrow\mathfrak{X}(M)$ is a map satisfying the following three properties:
\begin{gather}
    \conn_{f_1X + f_2Y} Z = f_1\conn_X Z + f_2 \conn_Y Z,\label{eqn:cov_deriv_lower_linear}\\
    \conn_X (c_1 Y  + c_2 Z) = c_1 \conn_X Y + c_2 \conn_X Z,\\
    \conn_X{fY} = (X[f])Y + f\conn_X Y,\label{eqn:cov_deriv_Leibniz}
\end{gather}
for any $X,Y,Z\in\mathfrak{X}(M)$, $f_1, f_2\in C^\infty(M)$, and $c_1,c_2\in\Re$. 
Also, $\conn_X Y$ is referred to as the \textit{covariant derivative} of $Y$ along $X$.

While both of the Lie derivative and covariant derivative map two vector fields into another, the Lie derivative $\mathcal{L}_X Y$ has the interpretation of representing the rate of change of $Y$ along the flow of $X$, and the covariant derivative $\conn_X Y$ compares vectors in different tangent spaces in a way consistent with the chosen geometry.
Further, the Lie derivative is intrinsically and uniquely formulated, but there are infinitely many covariant derivatives depending on the choice of the connection. 

A \emph{Riemannian metric} $\g$ on a manifold $M$ is a smooth, symmetric, positive-definite $(0,2)$-tensor field.
That is, for each point $p \in M$, the bilinear map $\g_p : T_p M \times T_p M \to \mathbb{R}$ defines an inner product on the tangent space at $p$, where the subscript $p$ indicates that the metric acts on $T_p M$.

A \emph{frame} on $U \subset M$ is a collection of $n$ smooth vector fields $\{E_1, \ldots, E_n\}$ on $U$ such that $\{E_1(p),\ldots,E_n(p)\}$ forms a basis for $T_p M$ at each $p\in U$. 
And, the frame is \emph{orthonormal} if $\g(E_i,E_j) = \delta_{ij}$ for all $i,j\in\{1,\ldots, n\}$.

Given a Riemannian metric, we can formulate the unique torsion-free connection $\nabla:\mathfrak{X}(M)\times\mathfrak{X}(M)\rightarrow\mathfrak{X}(M)$ that is compatible with the metric,  i.e., 
\begin{align}
    X[g(Y,Z)] = \g(\nabla_X Y, Z) + \g(Y, \nabla_X Z), \label{eqn:cov_deriv_comp}
\end{align}
for all $X,Y,Z\in\mathfrak{X}(M)$.
This is referred to as the \emph{Levi--Civita connection}.
As an affine connection, the Levi--Civita connection satisfies all three properties of \eqref{eqn:cov_deriv_lower_linear}--\eqref{eqn:cov_deriv_Leibniz}.
In the subsequent development, the \textit{covariant derivative} refers to the vector field obtained by the Levi--Civita connection.
For a vector field $Y=Y^k\partial_{k}$, its covariant derivative in coordinates is
\[
\nabla_{\partial_{i}}Y=(\partial_{i}Y^k+\Gamma^{k}_{ij}Y^j)\partial_{k},
\]
where $\Gamma^{k}_{ij}$ are the Christoffel symbols determined by
\begin{align}
    \Gamma^{k}_{ij}=\frac12\,\g^{k m}\big(\partial_i \g_{jm}+\partial_j \g_{im}-\partial_m \g_{ij}\big),\label{eqn:Gamma}
\end{align}
with $[\g^{ij}]=[\g_{ij}]^{-1}$.
Throughout this paper, we adopt the Einstein summation convention: 
whenever an index variable appears once in an upper position and once in a lower position within a single term, summation over that index is implied.

Next, we can define the following operators using the Riemannian metric.
For $f\in C^\infty(M)$, the \textit{gradient}, $\grad:C^\infty(M)\rightarrow\mathfrak{X}(M)$ is defined such that
\begin{align}
    \g(\grad\,f,X) =df (X), \label{eqn:grad}
\end{align}
for any $X\in\mathfrak{X}(M)$.
The \emph{divergence} $\div:\mathfrak{X}(M)\rightarrow C^\infty(M)$ of $X$ is the trace of the linear map $Y\rightarrow \nabla_Y X$~\cite[page 88]{lee2006riemannian}, given by
\begin{align}
\div X = \mathrm{tr} \big( Y \mapsto \nabla_Y X \big). \label{eqn:div}
\end{align}
The \textit{Laplace--Beltrami operator} $\Delta:C^\infty(M)\rightarrow C^\infty(M)$ is
\begin{align}
    \Delta f = \div(\grad f). \label{eqn:Laplace}
\end{align}
Finally, the \emph{Hessian} is the symmetric $(0,2)$-tensor defined by
\begin{align}
    \mathrm{Hess}_f(X,Y) = \g(\nabla_X(\grad f), Y)
    = X[Y[f]] - (\nabla_X Y)[f]. \label{eqn:Hess}
\end{align}

\subsection{Lie Group}

A \textit{Lie group} $G$ is a differentiable manifold that has group properties that are consistent with its manifold structure in the sense that the group operation is smooth. 
The \textit{group action} of a Lie group $G$ on a manifold $M$ is a smooth map $\phi: G \times M \to M$ satisfying $\phi(e, p) = p$ and $\phi(g, \phi(h, p)) = \phi(gh, p)$ for all $g, h \in G$ and $p \in M$, where $e$ is the identity element of $G$.\footnote{Throughout this paper, $\g$ denotes a Riemannian metric, and $g\in\G$ denotes an element of a Lie group $G$.}
The group acts on itself through \textit{left translation} $L_g(h) = gh$ and \textit{right translation} $R_g(h) = hg$. 

Next, \textit{Lie algebra} $\liealgebra$ is defined as the tangent space to the group at the identity element, i.e., $\mathfrak{g} = T_e G$, 
which is endowed with a bilinear  operation called the \textit{Lie bracket} $[\cdot, \cdot]: \mathfrak{g} \times \mathfrak{g} \to \mathfrak{g}$.
The Lie bracket is constructed as a specific case of the Jacobi--Lie bracket introduced in \eqref{eqn:Jac-Lie}.
For $\eta\in\liealgebra$, let $X_\eta(g)\in\mathfrak{X}(G)$ be a \textit{left-invariant vector field} obtained by the pushforward of $\eta$ by the left translation. 
More explicitly, 
\begin{align}
    X_\eta(g) = (L_g)_* \eta = (T_e L_g) (\eta) = g\eta,  \label{eqn:X_eta}
\end{align}
where $T_eL_g:\liealgebra\rightarrow T_g G$ is the tangential map of $L_g$, and the last expression is a common shorthand motivated by the fact that the left-invariant vector field is simply obtained by the product of $g$ and $\eta$ in a matrix Lie group. 
Throughout this paper, the tangential map of the left (or right) translation is simply written as $(T_e L_g)(\eta) = g\eta$ (or $(T_eR_g)(\eta)=\eta g$). 
The Lie bracket of $\eta, \zeta \in \liealgebra$ is then defined by the Jacobi--Lie bracket of their corresponding left-invariant vector fields, evaluated at the identity, i.e., 
\begin{align}
    [\eta, \zeta] = [X_\eta, X_\zeta](e).\label{eqn:bracket_g}
\end{align}

Given $g\in G$, the \textit{inner automorphism} ${I}_g:G\rightarrow G$ is defined as ${I}_g (h) = g h g^{-1}$. 
The \textit{adjoint operator} $\Ad_g :\liealgebra\rightarrow\liealgebra$ is the tangential map of ${I}_g(h)$ with respect to $h$ at $h=e$ along the direction $\eta\in\liealgebra$, i.e. $\Ad_g \eta = (T_e {I}_g) (\eta)$. 
For $\xi\in\liealgebra$, the \textit{ad operator} $\ad_\xi: \liealgebra\rightarrow\liealgebra$ is obtained by 
\begin{align*}
    \ad_\xi \eta = \frac{d}{d\epsilon}\bigg|_{\epsilon=0} \Ad_{\exp \epsilon \xi} \eta,
\end{align*}
which coincides with the Lie bracket, i.e. $\ad_\xi\eta =[\xi,\eta]$.

Let $\langle\cdot,\cdot\rangle$ denote the natural pairing between $\mathfrak g^*$ and $\mathfrak g$.
The \textit{coadjoint operator} $\Ad_g^*:\liealgebra^*\rightarrow\liealgebra^*$ is the dual of $\Ad_g$ defined by $\langle \Ad_g^* \alpha, \eta\rangle = \langle \alpha, \Ad_g \eta \rangle $ for any $\eta\in\liealgebra$ and $\alpha\in\liealgebra^*$. 
Similarly, \textit{co-ad operator} $\ad_\xi^*:\liealgebra^*\rightarrow\liealgebra^*$ is the dual of $\ad_\xi$ defined by $\langle \ad_\xi^* \alpha, \eta\rangle = \langle \alpha, \ad_\xi \eta \rangle $.


\subsection{Stochastic Differential Equation}

A Wiener process $W(t)$ is a continuous-time stochastic process characterized by three main properties: it starts at zero ($W(0)=0$), has independent increments, and the increment $W(t)-W(s)$ for $t>s$ is normally distributed with a zero mean and variance $t-s$. 
The non-differentiable nature of its sample paths motivates a special formulation of calculus.

An \emph{It\^{o} stochastic differential equation} (SDE) on $\mathbb{R}^n$ driven by an $m$-dimensional Wiener process has the form
\[
    d x = \tilde X(x)\, dt + \sum_{i=1}^m \sigma_i(x) \, dW_i,
\]
where $\tilde X, \sigma_1, \sigma_2, \ldots :\Re^n \to \Re^n$ are the \textit{drift} and the \textit{diffusion} coefficients.
The solution is interpreted in the It\^{o} sense by
\[
    x(t) = x(0) + \int_0^t \tilde X(x(s))\, ds + \sum_{i=1}^m \int_0^t \sigma_i(x(s))\, dW_i(s).
\]
where the stochastic integral is evaluated as the limit of left-point Riemann sums (Euler method in quadrature rules). 

Alternatively, when it is evaluated via the trapezoidal rule\footnote{Under additional assumptions, the Stratonovich integral may be defined according to the mid-point rule~\cite[pp. 126]{gliklikh2010global}.}, it is referred to as the \textit{Stratonovich} integral, and the corresponding SDE is denoted by
\[
    d x = X(x)\, dt + \sum_{i=1}^m \sigma_i(x) \circ dW_i,
\]
where $X, \sigma_1, \sigma_2\ldots:\Re^n\to\Re^n$. 

The It\^{o} integral and the Stratonovich integral are related by
\begin{align}
    \int_{0}^t \sigma(x)\circ dW = \int_{0}^t \sigma(x) dW + \frac{1}{2} \langle \sigma(x), W\rangle_t,\label{eqn:Ito_Stra_covariation}
\end{align}
where $\langle \cdot, \cdot \rangle_t$ denotes the quadratic covariation, defined by
\begin{align*}
    \langle \sigma(x), W\rangle_t  = \lim_{\|\Pi\|\to 0} \sum_{k=0}^{N-1} &
    \big(\sigma(x(t_{k+1}))-\sigma(x(t_k)) \big)\\
    & \times \big(W(t_{k+1})-W(t_k)\big),
\end{align*}
with $\Pi=\{0=t_0<\cdots<t_N=t\}$~\cite[pp. 126]{gliklikh2010global}.
Or, in \cite[pp. 100]{ikeda1989stochastic}, it has been shown that
\begin{align}
    A\circ dB = A\, dB + \frac{1}{2} \, dA\, dB, \label{eqn:Ito_Stra_d}
\end{align}
for stochastic processes $A$ and $B$.

The It\^{o} integral is defined for \textit{adapted} processes, implying that the integrand can only depend on information available up to the present time.
Specifically, let $z(t) = \int_0^t \sigma(z(\tau))dW$. 
Then $\mathbb{E}[z(t)|\mathcal{F}_s]=z(s)$ for $s<t$,
where $\mathcal{F}_s$ represents the \textit{filtration}, or the information available, at $s$. 
This is also referred to as that $z(t)$ is a \textit{martingale}.
Furthermore, this implies that
\begin{align}
\mathbb{E}\!\left[\int_s^t \sigma(z(\tau)) dW \,\middle|\, \mathcal{F}_s\right] = 0, \label{eqn:Ito_martingale}
\end{align}
or $\mathbb{E}[\sigma dW]=0$.
However, the It\^{o} stochastic integral does not satisfy the classical rules of calculus. This is due to the non-zero quadratic variation of the Wiener process, which gives rise to the following calculus rules:
\begin{align}
    dt^2 = 0, \quad dt\, dW = 0, \quad dW^2 = dt. \label{eqn:Ito_calculus}
\end{align}
These expressions are a shorthand used in derivations and are understood to hold in a mean-square sense (e.g., $\mathbb{E}[(dW)^2] = dt$). They are fundamental to It\^{o}'s Lemma, which implies $df(W) = f'(W) dW + \frac{1}{2}f''(W)dt$ for $f\in C^2(\Re)$.

In contrast, even when the integrand is adapted, the Stratonovich integral generally differs from a martingale by the finite-variation term in the It\^{o}–-Stratonovich conversion as given in \eqref{eqn:Ito_Stra_covariation}.
However, it satisfies the classical chain rule of calculus, e.g., $df(W) = f'(W)\circ dW$.

The \emph{infinitesimal generator} $\mathcal{A}$ of a stochastic process is the linear operator acting on $f \in C^\infty(\mathbb{R}^n)$ defined by
\begin{align}
\mathcal{A} f(x) = \lim_{t \downarrow 0} \frac{\mathbb{E}[f(x_t)|x_0 = x] - f(x)}{t}, \label{eqn:generator}
\end{align}
where the expectation is taken with respect to the process starting at $x_0=x$. 
Essentially, it describes the instantaneous rate of change of the expected value of $f$ given the initial state. 
While a single stochastic process can be described by multiple SDEs, depending on the choice of coordinates or the It\^{o} versus Stratonovich convention, it has a unique generator.

\section{Brownian Motion on a Manifold}\label{sec:BMM}

Brownian motion is defined as the stochastic process whose generator is the Laplace--Beltrami operator. 
The objective of this section is to formulate stochastic differential equations on a manifold that correspond to Brownian motion. 
We first derive an abstract expression for the generator of stochastic differential equations on a manifold, and then use it to formulate Brownian motion on a Riemannian manifold.
This is further extended to a Lie group. 

\subsection{Stochastic Differential Equation on a Manifold}

Consider a stochastic differential equation on a manifold $M$ of the form
\begin{align}
    dx = X(x)\,dt + \sum_{i=1}^m \sigma_i(x) \circ dW_i, \label{eqn:SDE}
\end{align}
where $X, \sigma_1, \ldots, \sigma_m \in \mathfrak{X}(M)$ are smooth vector fields and $W_i$ are independent real-valued Wiener processes. 
This is written in the \emph{Stratonovich} sense, which ensures that the solution flow preserves the manifold structure and that the usual rules of calculus hold. 
Under this formulation, the coordinate change of the stochastic trajectory follows the standard transformation of tangent vectors, as discussed in~\cite[pp.~139]{gliklikh2010global}.

We first identify the equivalent It\^{o} SDE on a manifold as follows. 
\begin{theorem}[It\^{o}--Stratonovich Conversion]\label{thm:Strat_Ito}
    The Stratonovich stochastic differential equation \eqref{eqn:SDE} is equivalent to the following It\^{o} stochastic differential equation:
    \begin{align}
        dx = \tilde X(x) dt + \sum_{i=1}^m \sigma_i (x) dW_i, \label{eqn:SDE_Ito}
    \end{align}
    where the modified drift vector field $\tilde X\in\mathfrak{X}(M)$ is
    \begin{align}
        \tilde X(x) = X(x) + \frac{1}{2} \sum_{i=1}^m \nabla_{\sigma_i(x)}\sigma_i(x). \label{eqn:X_tilde}
    \end{align}
\end{theorem}
\begin{proof}
    From \eqref{eqn:Ito_Stra_d},
    \begin{align}
        \sigma_i(x) \circ dW_i = \sigma_i(x) dW_i + \frac{1}{2} d\sigma_i(x)\, d W_i, \label{eqn:thm_Start_Ito_0}
    \end{align}
    where $d\sigma_i(x)$ corresponds to the stochastic differential of the vector field $\sigma_i$ when $x$ is a path of \eqref{eqn:SDE}.
Since the Stratonovich formulation satisfies the chain rule for a vector field,
\begin{align}
    d\sigma_i(x) = \nabla_{dx} \sigma_i = \nabla_{X} \sigma_i dt + \sum_{j=1}^m\nabla_{\sigma_j} \sigma_i \circ dW_j,\label{eqn:thm_Start_Ito_1}
\end{align}
where in the first equality, $dx$ in $\nabla_{dx} \sigma_i$ is interpreted as a linear combination of vector fields $X, \sigma_1, \sigma_2, \ldots $ weighted by $dt$ and $dW_i$ as in \eqref{eqn:SDE}, 
and the second equality follows from the linearity of the covariant derivative with respect to the bottom slot given by \eqref{eqn:cov_deriv_lower_linear}.

From \eqref{eqn:Ito_Stra_d}, we have
\begin{align}
(\nabla_{\sigma_j}\sigma_i \circ dW_j) dW_i & = \nabla_{\sigma_j} \sigma_i dW_j dW_i + \frac{1}{2} d(\nabla_{\sigma_j}\sigma_i) dW_j dW_i \nonumber \\
                                            & = \delta_{ij} \nabla_{\sigma_j} \sigma_i dt,\label{eqn:thm_Start_Ito_2}
\end{align}
where we have used the fact $dW_i dW_j = \delta_{ij} dt$ and that the higher-order terms of $dt,dW$ vanish.

Substituting \eqref{eqn:thm_Start_Ito_1} into \eqref{eqn:thm_Start_Ito_0}, and applying \eqref{eqn:thm_Start_Ito_2},
    \begin{align*}
        \sigma_i(x) \circ dW_i = \sigma_i(x) dW_i + \frac{1}{2} \nabla_{\sigma_i}\sigma_i dt.
    \end{align*}
    Combined with \eqref{eqn:SDE}, this yields \eqref{eqn:SDE_Ito} and \eqref{eqn:X_tilde}.
\end{proof}

Suppose $M=\Re^1$ and $m=1$.
Then, the correction term in \eqref{eqn:X_tilde} is given by
\[
\frac12 \nabla_{\sigma}\sigma
= \frac12 \parenth{\sigma \deriv{\sigma}{x} + \Gamma^{1}_{11}\sigma^2} = \frac12 \sigma \deriv{\sigma}{x},   
\]
where we used the fact that the Christoffel symbols vanish in $M=\Re$.
The above recovers the common It\^{o}--Stratonovich conversion term in the Euclidean space.
In other words, the conversion term for a stochastic process on a manifold includes an additional term contributed by the Christoffel symbol in the covariant derivative of the manifold.

Next, the generator of this process is obtained as follows. 
\begin{theorem}[Generator of SDE]\label{thm:generator}
    The infinitesimal generator $\mathcal{A}:C^\infty(M)\rightarrow C^\infty(M)$ of the Stratonovich stochastic differential equation \eqref{eqn:SDE} on $f\in C^\infty(M)$ is given by
    \begin{align}
        \mathcal{A} f & = X[f] + \frac{1}{2}\sum_{i=1}^m \sigma_i [\sigma_i[f]] .\label{eqn:A_Strat}
    \end{align}
    Or equivalently, the generator of the corresponding It\^{o} stochastic differential equation \eqref{eqn:SDE_Ito} is
    \begin{align}
        \mathcal{A} f & = \tilde X[f] + \frac{1}{2}\sum_{i=1}^m \{\sigma_i [\sigma_i[f]] - (\nabla_{\sigma_i} \sigma_i) [f]\} \label{eqn:A_Ito}\\
                      & = \tilde X[f] + \frac{1}{2}\sum_{i=1}^m \mathrm{Hess}_f (\sigma_i, \sigma_i). \label{eqn:A_Ito_Hess}
    \end{align}
\end{theorem}
\begin{proof}
    Let $f\in C^\infty(M)$.
    Since the Stratonovich stochastic differential equation satisfies the chain rule in its standard first-order form, from \eqref{eqn:SDE},
    \begin{align}
        df(x) & = X[f] dt + \sum_{i=1}^m \sigma_i[f] \circ dW_i. \label{eqn:df_generator}
    \end{align}
    In other words, there is no need to expand $df$ to second order as in the It\^{o} calculus.
    However, since the Stratonovich formulation does not have the martingale property for $f(x)$ without conversion to the It\^{o} form, i.e., \eqref{eqn:Ito_martingale} is not satisfied, we cannot simply take the mean to eliminate the second term. 

    Instead, we convert \eqref{eqn:df_generator} into an It\^{o} form using \eqref{eqn:Ito_Stra_d} to obtain
    \begin{align}
        df(x) & = X[f] dt + \sum_{i=1}^m \sigma_i[f] dW_i + \frac{1}{2}\sum_{i=1}^m d(\sigma_i[f]) dW_i. \label{eqn:df_A0}
    \end{align}

    Now, we simplify the last term. 
    From Cartan's magic formula~\cite{marsden2013introduction}, we have $X[\alpha] = \mathcal{L}_X \alpha = i_X d\alpha + d(i_X\alpha)$ for any $X\in\mathfrak{X}(M)$ and a differentiable form $\alpha$, where $\mathcal{L}$ and $i$ denote the Lie derivative and the interior product, respectively. 
    Setting $\alpha =df$, we obtain $X[df] = i_X d^2f + d (i_X df)$. 
    Since $d^2 f= 0$ and $i_X df = df(X) = X[f]$, this reduces to $X[df] = d(X[f])$, i.e., the Lie-derivative and the differential commute.
    Thus, the following factor in the last term of \eqref{eqn:df_A0} is
    \begin{align*}
        d(\sigma_i[f]) & = \sigma_i [df] \\
                       & = \sigma_i[X[f]]dt + \sum_{j=1}^m \sigma_i [\sigma_j[f]]\circ dW_j.
    \end{align*}
    Substituting this into \eqref{eqn:df_A0}, and following the same development as presented in \eqref{eqn:thm_Start_Ito_2}, 
    \begin{align*}
        df(x) & = X[f] dt + \sum_{i=1}^m \sigma_i[f] dW_i + \frac{1}{2}\sum_{i=1}^m \sigma_i[\sigma_i[f]] dt.
    \end{align*}
    Since $\E[dW_i]=0$, from \eqref{eqn:generator}, this yields \eqref{eqn:A_Strat}.
    Finally, \eqref{eqn:A_Ito} is obtained by the It\^{o}--Stratonovich conversion \eqref{eqn:X_tilde}, and \eqref{eqn:A_Ito_Hess} is obtained by the definition of the Hessian at \eqref{eqn:Hess}.
\end{proof}

Consider a special case of \Cref{thm:generator} when $M=\Re^1$.
We have
\begin{align*}
    \sigma[\sigma[f]] & = \sigma\deriv{}{x}\parenth{\sigma\deriv{f}{x}} = \sigma\deriv{\sigma}{x}\deriv{f}{x} + \sigma^2 \frac{\partial^2 f}{\partial x^2},\\
    (\nabla_\sigma \sigma) [f] & = \sigma\deriv{\sigma}{x}\deriv{f}{x}.
\end{align*}
Thus, \eqref{eqn:A_Ito} reduces to
\begin{align*}
    \mathcal{A} f = \tilde X \deriv{f}{x} + \frac12 \sigma^2 \frac{\partial^2 f}{\partial x^2},
\end{align*}
which is the well-known expression for the generator of the It\^{o} process on $\Re^1$. 

In \eqref{eqn:df_generator} of the proof, we expanded $df$ up to the first order to utilize the Stratonovich SDE. 
Alternatively, we can expand it up the second order and apply the It\^{o} SDE. 
This yields the It\^{o} lemma on a manifold as summarized below.

\begin{theorem}[It\^{o} Lemma]\label{thm:Ito}
    Suppose $x$ is driven by \eqref{eqn:SDE_Ito}, and let $f\in C^\infty(M)$.
    Then, we have
    \begin{align}
        df(x) = \parenth{\tilde X[f] + \frac{1}{2}\sum_{i=1}^m \mathrm{Hess}_f (\sigma_i,\sigma_i)}\! dt + \sum_{i=1}^m \sigma_i[f] dW_i. \label{eqn:Ito_lem}
    \end{align}
\end{theorem}
\begin{proof}
    Let $f\in C^\infty(M)$ be expanded up to the second order by
    \begin{align}
        df = \nabla_{dx} f + \frac{1}{2}\mathrm{Hess}_f (dx, dx), \label{eqn:df_Ito_lem_0}
    \end{align}
    where we adopted the covariant Taylor expansion~\cite{avramidi2007analytic} to make \eqref{eqn:df_Ito_lem_0} coordinate invariant. 
    Since this is of the second order, we can substitute the It\^{o} SDE \eqref{eqn:SDE_Ito} into \eqref{eqn:df_Ito_lem_0}.
    The first term of \eqref{eqn:df_Ito_lem_0} is
    \begin{align*}
        \nabla_{dx} f = X[f] dt + \sum_{i=1}^m \sigma_i[f] dW_i,
    \end{align*}
    Since the Hessian is bilinear and $dt^2 = dtdW=0$, the second term of \eqref{eqn:df_Ito_lem_0}, after scaled by the factor 2, is
    \begin{align*}
        \mathrm{Hess}_f (dx, dx) & = \mathrm{Hess}_f \parenth{\sum_{i=1}^m \sigma_i dW_i, \sum_{j=1}^m \sigma_j dW_j} \\
                                 & = \sum_i \mathrm{Hess}_f (\sigma_i, \sigma_i) dt.
    \end{align*}
    Substituting these into \eqref{eqn:df_Ito_lem_0}, we obtain \eqref{eqn:Ito_lem}.
\end{proof}

Again we verify the results of \Cref{thm:Ito} on the Euclidean space. 
When $M=\Re^1$, \eqref{eqn:Ito_lem} reduces to
\begin{align*}
    df(x) = \parenth{\tilde X \deriv{f}{x} + \frac{1}{2}\sigma^2\frac{\partial^2 f}{\partial x^2}}dt + \sigma\deriv{f}{x}dW,
\end{align*}
which recovers the It\^{o} lemma on $\Re$. 

Also, by taking the mean of \eqref{eqn:Ito_lem}, we can easily derive \eqref{eqn:A_Ito_Hess}.
As such, \eqref{eqn:Ito_lem} can serve as an alternative proof for \Cref{thm:generator}.

\subsection{Brownian Motion}\label{sec:BM_Ei}

Brownian motion on $M$ is defined as a stochastic process whose generator is one-half of the Laplace--Beltrami operator~\cite{hsu2002stochastic}.
Intuitively, the process governed by the Laplace--Beltrami operator describes diffusion, representing the net movement of particles from regions of higher concentration to regions of lower concentration.
More precisely, if $\Delta f$ is positive at a point, then the point is a local minimum of $f$, and the value of $f$ tends to increase due to the action of the generator; conversely, a negative value of $\Delta f$ corresponds to a local maximum, where $f$ tends to decrease.
Thus, under this dynamics, the distribution of $f$ tends to become increasingly \emph{flat} over time.
It is intrinsically formulated on a manifold, independent of the choice of coordinates. 
Further, it results in an isometric diffusion, without any \textit{preferred} direction, since the Laplace--Beltrami operator commutes with isometries such as rotations or translations.


Our construction of the SDE relies on the orthonormal frame introduced in \Cref{sec:diff_geo}.
We first show the following identities for the gradient, divergence, and Laplace--Beltrami operator expressed in terms of the orthonormal frame. 

\begin{prop}\label{prop:grad_div_laplace_Ei}
    Let $\{E_1,\ldots, E_n\}$ be an orthonormal frame, i.e., a set of vector fields $E_i\in\mathfrak{X}(M)$ that forms an orthonormal basis for the tangent space, satisfying $\g(E_i,E_j)=\delta_{ij}$. 
    The gradient $\grad:C^\infty(M)\rightarrow\mathfrak{X}(M)$, the divergence $\div: \mathfrak{X}(M)\rightarrow C^\infty (M)$, and the Laplace--Beltrami operator  $\Delta:C^\infty(M)\rightarrow C^\infty(M)$ are expressed in terms of the orthonormal frame as
    \begin{align}
        \grad f & = \sum_{i=1}^n (E_i[f]) E_i, \label{eqn:grad_Ei}\\
        \div X & = \sum_{i=1}^n \g(\nabla_{E_i} X, E_i), \label{eqn:div_Ei}\\
        \Delta f & = \sum_{i=1}^n \left( E_i[ E_i [f] ] - (\nabla_{E_i} E_i) [f] \right), \label{eqn:laplace_Ei}\\
                 & = \sum_{i=1}^n \mathrm{Hess}_f (E_i, E_i),  \label{eqn:laplace_Ei_Hess}
    \end{align}
    for $f\in\mathcal{C}^\infty (M)$ and $X\in\mathfrak{X}(M)$.
    Further, these are independent of the choice of the orthonormal frame.
\end{prop}
\begin{proof}
    Let $\grad f$ be resolved in $\{E_i\}$, i.e., $\grad f = \sum_{i} c^i(x) E_i$ for $c^i\in C^\infty(M)$.
    Substituting this into \eqref{eqn:grad}, and setting $X = E_j$,
    \begin{align*}
        \g(\grad f, E_j) = \sum_i  c^i \g(E_i, E_j) = c^j = E_j[ f ],
    \end{align*}
    where we have used the linearity of the metric and $\g(E_i,E_j)=\delta_{ij}$. 
    Substituting this coefficient back to $\grad f = \sum_{i} c^i(x) E_i$, we obtain \eqref{eqn:grad_Ei}.

    Next, we show \eqref{eqn:grad_Ei} is independent of the choice of the frame. 
    Let $\{F_i\}$ be another orthonormal frame. 
    Then, there exists an orthonormal matrix $R:M\rightarrow \mathrm{O}(n)$ such that
    \begin{align}
        F_i = \sum_j R_{ij} E_j,\label{eqn:F_i}
    \end{align}
    at each $x\in M$, where $R_{ij}\in\Re$ denotes the elements of $R$. 
    We evaluate the right-hand side of \eqref{eqn:grad_Ei} using $\{F_i\}$ to obtain
    \begin{align*}
        \sum_i (F_i[f])F_i & = \sum_{i} \bigg\{\bigg(\sum_{j} R_{ij}E_{j}[f]\bigg) \bigg(\sum_k R_{ik} E_k\bigg)\bigg\} \\
                           & =  \sum_{j,k} \parenth{\sum_i R_{ij} R_{ik}} (E_j [f]) E_k.
    \end{align*}
    But using the orthogonality of $R$, namely $R^TR = I_{n\times n}$, we have $\sum_i R_{ij}R_{ik} = \delta_{jk}$.
    Substituting this, the above expression reduces to \eqref{eqn:grad_Ei}.
    This verifies that \eqref{eqn:grad_Ei} is independent of the choice of the frame. 

    Next, using the definition of the trace of a linear operator, we obtain \eqref{eqn:div_Ei} directly from \eqref{eqn:div}. 
    Similarly, we evaluate the right-hand side of \eqref{eqn:div_Ei} with $\{F_i\}$ as follows:
    \begin{align*}
        \sum_i \g(\nabla_{F_i} X, F_i) = \sum_i \g( \nabla_{\sum_j R_{ij} E_j} X, \sum_k R_{ik}E_k).
    \end{align*}
    Using the linearity of the covariant derivative in the lower slot \eqref{eqn:cov_deriv_lower_linear}, and the linearity of the inner product, we can pull out the coefficients of $R$ to obtain
    \begin{align*}
        \sum_{i,j,k} R_{ij} R_{ik} \g( \nabla_{E_j} X, E_k)=
        \sum_{j,k} \bigg(\sum_i R_{ij} R_{ik}\bigg) \g( \nabla_{E_j} X, E_k),
    \end{align*}
    which reduces to \eqref{eqn:div_Ei} using the orthogonality of $R$.
    Therefore, \eqref{eqn:div_Ei} is also independent of the choice of the frame. 

    Finally, the Laplace--Beltrami operator is formulated by $\Delta f = \div (\grad f)$. 
    Substituting \eqref{eqn:grad_Ei} into \eqref{eqn:div_Ei}, and using \eqref{eqn:cov_deriv_lower_linear},
    \begin{align*}
        \Delta f = \sum_{i,j} \g (\nabla_{E_i} ((E_j[f]) E_j), E_i).  
    \end{align*}
    According to the product rule of the covariant derivative \eqref{eqn:cov_deriv_Leibniz}, 
    \begin{align*}
        \nabla_{E_i} (E_j[f]) E_j = (E_i[E_j[f]]) E_j + (E_j[f]) \nabla_{E_i} E_j.
    \end{align*}
    Substituting this into the above, and using $\g(E_i, E_j) = \delta_{ij}$,
    \begin{align}
        \Delta f = \sum_i E_i[E_i[f]] + \sum_{i,j} (E_j[f]) \g(\nabla_{E_i} E_j, E_i). \label{eqn:delta_f_0}
    \end{align}
    In the metric compatibility of the covariant derivative given by \eqref{eqn:cov_deriv_comp}, we set
    $X=Z=E_i$ and $Y=E_j$ to obtain
    \begin{align*}
        0 = \g(\nabla_{E_i} E_j, E_i) + \g(E_j, \nabla_{E_i} E_i). 
    \end{align*}
    Using this, the second term of \eqref{eqn:delta_f_0} is rearranged into
    \begin{align*}
        -\sum_{i,j} E_j[f] \g(E_j, \nabla_{E_i}{E_i})  & = -\sum_i \g (\textstyle \sum_j (E_j[f]) E_j, \nabla_{E_i}E_i )\\
                                                      & = -\sum_i \g (\grad f, \nabla_{E_i}E_i )\\
                                                      & = - \sum_i (\nabla_{E_i}E_i)[f],
    \end{align*}
    where the second equality is from \eqref{eqn:grad_Ei}, and the third is from \eqref{eqn:grad}.
    Substituting this back to \eqref{eqn:delta_f_0} yields \eqref{eqn:laplace_Ei}.
    Since \eqref{eqn:laplace_Ei} is obtained by the composition of \eqref{eqn:grad_Ei} and \eqref{eqn:div_Ei} that are independent of the choice of the frame, \eqref{eqn:laplace_Ei} is independent of the choice of the frame.
    Finally, \eqref{eqn:laplace_Ei_Hess} is obtained by \eqref{eqn:Hess}.
\end{proof}

The resemblance between the Laplace--Beltrami operator given by \eqref{eqn:laplace_Ei} and the expressions of the generator in \Cref{thm:generator} motivates the following formulation of an intrinsic stochastic differential equation for the Brownian motion. 
The key idea is to inject noise along each axis $E_i$ of an orthonormal frame using a Wiener process, and choose the remaining drift term such that the corresponding generator becomes the Laplace--Beltrami operator.

\begin{theorem}[Brownian Motion on a Manifold]\label{thm:Brownian_motion}
    Let $\{E_1,\dots,E_n\}$ be an orthonormal frame on $M$.  
    Consider the following Stratonovich stochastic differential equation
    \begin{align}
        dx &= -\frac12 \sum_{i=1}^n \nabla_{E_i} E_i \, dt + \sum_{i=1}^n E_i \circ dW_i,
        \label{eqn:SDE_Brownian}
    \end{align}
    and its equivalent It\^{o} form
    \begin{align}
        dx &= \sum_{i=1}^n E_i  dW_i.
        \label{eqn:SDE_Brownian_Ito}
    \end{align}
    Then the infinitesimal generator of either \eqref{eqn:SDE_Brownian} or \eqref{eqn:SDE_Brownian_Ito} is 
    \[
        \mathcal{A} f = \frac12 \Delta f,
    \]
    such that the resulting process is the Brownian motion on $M$.
\end{theorem}
\begin{proof}
    First, we show that \eqref{eqn:SDE_Brownian_Ito} is equivalent to \eqref{eqn:SDE_Brownian}.
    By \Cref{thm:Strat_Ito} with $\sigma_i = E_i$, the drift term in \eqref{eqn:SDE_Brownian} is precisely
\[
    -\frac12\sum_{i=1}^n \nabla_{E_i} E_i
    = -\frac12\sum_{i=1}^n \nabla_{\sigma_i} \sigma_i,
\]
which cancels the It\^{o}--Stratonovich correction term of \eqref{eqn:X_tilde}, yielding the drift-free It\^{o} form \eqref{eqn:SDE_Brownian_Ito}.

Next, from \eqref{eqn:A_Strat} and Proposition~\ref{prop:grad_div_laplace_Ei}, the generator of \eqref{eqn:SDE_Brownian} is
\[
    \mathcal{A} f
    = -\frac12 \sum_{i=1}^n \nabla_{E_i} E_i[f]
      + \frac12 \sum_{i=1}^n E_i(E_i[f]),
\]
which corresponds to $\frac12 \Delta$ from \eqref{eqn:laplace_Ei}.
\end{proof}

When $M=\Re^1$, \eqref{eqn:SDE_Brownian} or \eqref{eqn:SDE_Brownian_Ito} reduce to
\begin{align*}
    dx = dW,
\end{align*}
corresponding to the standard Brownian motion on $\Re^1$.

\begin{remark}
The given stochastic differential equations \eqref{eqn:SDE_Brownian} and \eqref{eqn:SDE_Brownian_Ito} depend on the choice of the orthonormal frame $\{E_i\}$. 
For example, if the It\^{o} SDE \eqref{eqn:SDE_Brownian_Ito} is reformulated via the frame $\{F_i\}$ given by \eqref{eqn:F_i}, we obtain
\begin{align}
    d x = \sum_{i,j} R_{ij} E_j dW_i. \label{eqn:SDE_Brownian_Ito_2}
\end{align}
However, as discussed in \cite[Section 4.8.1]{chirikjian2009stochastic}, this does not matter. 
Although \eqref{eqn:SDE_Brownian_Ito_2} differs from \eqref{eqn:SDE_Brownian_Ito} pathwise, 
the two SDEs have the same generator $\frac12\Delta$, and thus define the same diffusion process in law.  
This invariance follows from \Cref{prop:grad_div_laplace_Ei}, where the expression of $\Delta$ given by \eqref{eqn:laplace_Ei} is independent of the choice of orthonormal frame.
\end{remark}
\begin{remark}
    The drift term $-\frac{1}{2}\sum_i \nabla_{E_i} E_i$ in the Stratonovich form \eqref{eqn:SDE_Brownian} 
    measures how the basis vectors of the tangent space fail to remain parallel under the Levi--Civita connection, reflecting the local spreading or contraction of the frame due to curvature.  
    This is precisely the curvature-induced correction that ensures the resulting diffusion process has generator $\tfrac{1}{2}\Delta$.
\end{remark}

The given expression \eqref{eqn:SDE_Brownian_Ito} for the It\^{o} SDE for Brownian motion may appear to be counterintuitive, as there is no correction term. 
But, we can be assured that the resulting generator, provided that it is properly computed as an operator on the manifold, is indeed given by the Laplace--Beltrami operator according to \Cref{thm:generator}, where the covariant derivative in \eqref{eqn:A_Ito} provides the required correction terms. 

However, if the It\^{o} formulation \eqref{eqn:SDE_Brownian_Ito} is expressed in the embedding Euclidean space, an additional drift term should be included to ensure that the process evolves on the manifold. 
As the drift term is normal to the tangent space at each point of the manifold, it does not contribute to the generator.
Consequently, the generator of the resulting augmented form of \eqref{eqn:SDE_Brownian_Ito} is still the Laplace--Beltrami operator. 
This will be more explicitly discussed in \Cref{sec:SO3}.

\subsection{Brownian Motion on a Lie Group}\label{sec:BM_G}

In this section, we consider a special case when the manifold $M$ is an $n$-dimensional Lie group $G$.
In \Cref{thm:Brownian_motion}, we considered an arbitrary orthonormal frame $\{E_1, \ldots E_n\}$, which is often formulated locally.
For a Lie group, we can construct a specific choice of a Riemannian metric and an orthonormal frame in a global fashion as follows. 

As the Lie algebra $\liealgebra$ is an $n$-dimensional vector space, there are an inner product $\langle\cdot, \cdot\rangle_\liealgebra:\liealgebra \times \liealgebra\rightarrow\Re$, and an orthonormal basis of $\liealgebra$ given by
\begin{align}
    \{ e_1^\liealgebra, e_2^\liealgebra, \ldots, e_n^\liealgebra\} \in \liealgebra^n, \label{eqn:basis_algebra}
\end{align}
satisfying $\langle e_i^\liealgebra, e_j^\liealgebra \rangle_\liealgebra = \delta_{ij}$. 

The inner product on $\liealgebra$ is extended to a left-invariant Riemannian metric on $G$ via the left-translation. 
More specifically, for tangent vectors $v, w\in T_gG$, the metric of $v$ and $w$ is obtained by
\begin{align}
    \g(v, w) = \langle (L_{g^{-1}})_* v, (L_{g^{-1}})_* w \rangle_\liealgebra = \langle g^{-1}v, g^{-1}w \rangle_\liealgebra, \label{eqn:metric_G}
\end{align}
which is left-invariant by construction, i.e., $(L_g)^* \g = \g$ for any $g\in G$. 

Further, an orthonormal basis of the tangent space can be constructed by translating the basis \eqref{eqn:basis_algebra} of $\liealgebra$ into $T_gG$ according to
\begin{align}
    E_i = (L_g)_* e_i^{\liealgebra} = g e_i^{\liealgebra} \in T_gG. \label{eqn:frame_G}
\end{align}
Since \eqref{eqn:frame_G} is constructed by a left translation of the orthonormal basis \eqref{eqn:basis_algebra}, and the metric \eqref{eqn:metric_G} is left invariant, it is straightforward to show \eqref{eqn:frame_G} is an orthonormal frame of $G$, i.e., $\g(E_i,E_j)=\langle g^{-1} E_i, g^{-1} E_j\rangle_\liealgebra =\langle e_i^\liealgebra, e_j^\liealgebra \rangle_\liealgebra =  \delta_{ij}$. 

Another desirable feature is that the orthonormal frame \eqref{eqn:frame_G} is defined globally on $G$. 
As such, we can construct stochastic differential equations for a Brownian motion on a Lie group using \eqref{eqn:frame_G} as follows. 
\begin{theorem}[Brownian Motion on a Lie Group]\label{thm:Brownian_motion_G}
    Let $G$ be a Lie group endowed with the left-invariant Riemannian metric given by \eqref{eqn:metric_G}, 
    where we identify $\liealgebra^*$ with $\liealgebra$ using the inner product on $\liealgebra$. 
    Let $\{e^\liealgebra_1,\dots,e^\liealgebra_n\}$ be an orthonormal basis of $\liealgebra$, constructed by \eqref{eqn:basis_algebra}.

    Consider the following Stratonovich stochastic differential equation
    \begin{align}
        g^{-1}dg &= \frac12 \sum_{i=1}^n \ad^*_{e_i^\liealgebra} e_i^\liealgebra \, dt + \sum_{i=1}^n e_i^\liealgebra \circ dW_i,
        \label{eqn:SDE_Brownian_G}
    \end{align}
    and its equivalent It\^{o} form
    \begin{align}
        g^{-1} dg &= \sum_{i=1}^n e_i^\liealgebra  dW_i.
        \label{eqn:SDE_Brownian_Ito_G}
    \end{align}
    Then the infinitesimal generator of either \eqref{eqn:SDE_Brownian_G} or \eqref{eqn:SDE_Brownian_Ito_G} is 
    \[
        \mathcal{A} f = \frac12 \Delta f,
    \]
    such that the resulting process is the Brownian motion on $G$.
\end{theorem}
\begin{proof}
    Let $X,Y,Z\in\mathfrak{X}(M)$ be vector fields on a Riemannian manifold $M$.
    From the fundamental theorem of Riemannian geometry (or, the Koszul formula)~\cite[(5.1)]{lee2006riemannian}, the covariant derivative is related to the Jacobi--Lie bracket according to
    \begin{align}
        2\,\g(\nabla_X Y, Z) 
&= X\!\big[\g(Y,Z)\big] + Y\!\big[\g(Z,X)\big] - Z\!\big[\g(X,Y)\big] \nonumber\\
&\quad - \g(X,[Y,Z]) - \g(Y,[X,Z]) + \g(Z,[X,Y]) . \label{eqn:Koszul}
    \end{align}
    Suppose that the manifold is a Lie group $G$ and $X,Y,Z$ are left-invariant vector fields $G$ obtained by
    \[
        X=(L_g)_*\eta=g\eta,\quad Y=(L_g)_*\xi=g\xi,\quad Z=(L_g)_*\zeta=g\zeta,
    \]
    for $\eta,\xi,\zeta\in\mathfrak g$.
    Since the metric $\g$ of \eqref{eqn:metric_G} is left-invariant,
    $\g(Y,Z),\g(Z,X)$ and $\g(X,Y)$ become constants.
    For example, $\g(Y,Z)= \langle \xi, \zeta \rangle_\liealgebra$.
    Thus, first three terms of the right-hand side of \eqref{eqn:Koszul} correspond to the Lie-derivatives of a constant, and therefore, vanish. 

    Further, since $[X,Y]=(L_g)_*[\eta,\xi]$ from \eqref{eqn:bracket_g}, and using the left-invariance of $\g$, the remaining terms of \eqref{eqn:Koszul} are rearranged into
    \begin{align*}
        2 & \langle (L_{g^{-1}})_*(\nabla_X Y), \zeta \rangle_{\mathfrak g}\\
                             &= - \langle \eta,[\xi ,\zeta]\rangle_{\mathfrak g} 
                             - \langle \xi ,[\eta ,\zeta]\rangle_{\mathfrak g}  
                             + \langle \zeta, [\eta ,\xi]  \rangle_{\mathfrak g} \\
                             &= - \langle \operatorname{ad}_\xi^* \eta, \zeta \rangle_{\mathfrak g}
                             - \langle \operatorname{ad}_\eta^* \xi, \zeta \rangle_{\mathfrak g}
                             + \langle [\eta, \xi], \zeta \rangle_{\mathfrak g},
    \end{align*}
    where we identified $\liealgebra^*$ with $\liealgebra$ with the inner product on $\liealgebra$ in the last equality.  
    Since this holds for any arbitrary $\zeta\in\mathfrak g$, we obtain the following identity to compute the covariant derivative of vector fields on a Lie group using the bracket and the co-ad operator:
    \[
        (\nabla_X Y)(g) = \frac12\,(L_g)_*\!\left([\eta,\xi]- \operatorname{ad}_\eta^* \xi - \operatorname{ad}_\xi^* \eta\right).
    \]
    In particular, setting $X=Y=E_i=(L_g)_* e_i^{\mathfrak g}$ yields
    \begin{equation}\label{eqn:nabla_Ei_Ei_G}
        \nabla_{E_i} E_i = -\, (L_g)_*(\ad_{e_i^{\liealgebra}}^{*} e_i^{\liealgebra}) = -g (\ad_{e_i^{\liealgebra}}^{*} e_i^{\liealgebra}).
    \end{equation}
    Substituting this and \eqref{eqn:frame_G} into \eqref{eqn:SDE_Brownian}, we obtain \eqref{eqn:SDE_Brownian_G}.
    Next, \eqref{eqn:SDE_Brownian_Ito_G} is obtained by substituting \eqref{eqn:frame_G} into \eqref{eqn:SDE_Brownian_Ito}.
\end{proof}

Comparing \eqref{eqn:SDE_Brownian_G} with \eqref{eqn:SDE_Brownian}, the sum of covariant derivatives in the drift vector field of \eqref{eqn:SDE_Brownian} is replaced by the co-ad operator. 
In fact, we can further show that the drift term of \eqref{eqn:SDE_Brownian_G} vanishes for any unimodular group as summarized below. 

\begin{corollary}[Brownian Motion on a Unimodular Group]\label{cor:Brownian_motion_G_unimodular}
    If the Lie group $G$ is unimodular, the stochastic differential equation \eqref{eqn:SDE_Brownian_G} in a Stratonovich form for Brownian motion reduces to
    \begin{align}
        g^{-1}dg &= \sum_{i=1}^n e_i^\liealgebra \circ dW_i,
        \label{eqn:SDE_Brownian_G_unimodular}
    \end{align}
\end{corollary}
\begin{proof}
    Let the drift term in \eqref{eqn:SDE_Brownian_G} be $\frac{1}{2}J\in\liealgebra$, i.e., 
    \begin{align}
        J = \sum_{i=1}^n \ad^*_{e_i^\liealgebra} e_i^\liealgebra. \label{eqn:J}
    \end{align}
    For $\eta\in\g$, using the duality and antisymmetry of the ad operator,
    \begin{align}
        \langle J, \eta \rangle_\liealgebra
        & = \sum_{i=1}^n \langle \ad^*_{e_i^\liealgebra} e_i^\liealgebra, \eta \rangle_\liealgebra
        = - \sum_{i=1}^n \langle e_i^\liealgebra, \ad_{\eta} e_i^\liealgebra \rangle_\liealgebra \nonumber\\
        & = -\mathrm{tr}[\ad_\eta],\label{eqn:Jeta}
    \end{align}
    where the last equality is from the common formulation for the trace of a linear operator. 
    It has been shown that for a unimodular group $\mathrm{tr}[\ad_\eta]=0$ for any $\eta\in\liealgebra$~\cite[(10.65)]{chirikjian2011stochastic,encmath2014},
    which implies that $\langle J,\eta \rangle_\liealgebra=0$ for any $\eta\in\liealgebra$.
    Therefore, we have $J=0$. 
    Substituting this into \eqref{eqn:SDE_Brownian_G} yields \eqref{eqn:SDE_Brownian_G_unimodular}.
\end{proof}

\begin{remark}\label{rem:unimodular}
    \Cref{cor:Brownian_motion_G_unimodular} states that for any unimodular Lie group, the drift-free form of the Stratonovich stochastic differential equation \eqref{eqn:SDE_Brownian_G_unimodular} can be utilized for Brownian motion.
    A unimodular group is a group with a Haar measure that is both left and right invariant~\cite[Ch. 14]{chirikjian2011stochastic}, and a large class of Lie groups in science and engineering is unimodular, including:
    \begin{itemize}
        \item Abelian groups (e.g., $\Sph^1$, $\mathbb{T}^n$)
        \item Compact groups (e.g., $SO(n)$, $SU(n)$, $Sp(n)$, $O(n)$)
        \item Semisimple groups (e.g., $SL(n,\Re)$ for $n\geq 2$)
        \item Nilpotent groups (e.g., the Heisenberg group $H_n$)
        \item Euclidean groups ($SE(n)$ and $E(n)$)
        \item The general linear group $GL(n,\Re)$
    \end{itemize}
    Moreover, any Lie group that admits a bi-invariant Riemannian metric is unimodular.
    Therefore, \Cref{cor:Brownian_motion_G_unimodular} applies broadly, including to any Lie group with a bi-invariant metric.
    The affine group $\mathsf{Aff}(\Re)$ is a notable example of a non-unimodular Lie group (see \Cref{sec:Affine}).
\end{remark}

\begin{remark}
    It should be noted that the results of \Cref{thm:Brownian_motion_G} and \Cref{cor:Brownian_motion_G_unimodular} are based on the left-invariant Riemannian metric \eqref{eqn:metric_G}, and the left-invariant frame \eqref{eqn:frame_G} of a Lie group.
    If another metric or frame is utilized, then the group can be viewed as a Riemannian manifold and apply the general result of \Cref{thm:Brownian_motion}.
\end{remark}

\section{Brownian Motion on an Embedded Manifold}\label{sec:BMEM}

This section formulates stochastic differential equations for Brownian motion on a manifold embedded in a higher-dimensional Euclidean space. 
As illustrated in \Cref{sec:ex}, a coordinate representation of a stochastic differential equation may be valid only locally. 
This limitation can be addressed by formulating the stochastic differential equation on the ambient Euclidean space.
However, this requires careful construction to ensure the resulting process evolves on the manifold. 
Our approach is to design a stochastic differential equation whose generator on the ambient space coincides with the generator of Brownian motion on the embedded manifold.

\subsection{Embedded Manifold}

Consider an $n$-dimensional Riemannian manifold $M$, embedded in the $\bar n$-dimensional Euclidean space with $\bar n > n$. 
In fact, according to the Whitney embedding theorem~\cite[Chapter 6]{lee2003smooth}, any $n$-dimensional manifold admits a smooth embedding into a Euclidean space $\Re^{2n}$. 

The metric on $M$ is induced from an inner product on $\Re^{\bar n}$. 
As such, the metric is often denoted by a pairing, i.e., $\g(\cdot, \cdot) = \pair{\cdot, \cdot}$, in this section. 
At any point $x\in M$, the tangent space $T_x M$ is an $n$-dimensional subspace of $\Re^{\bar n}$. 
As presented in \Cref{sec:BM_Ei}, let $\{E_i(x)\}$ be an orthonormal frame for the tangent space $T_x M$. 
Using this, we can define an orthogonal projection operator $P(x): \Re^{\bar n}\rightarrow T_xM$ from the ambient $\Re^{\bar n}$ to the tangent space $T_xM$ by
\begin{align}
    P(x) v = \sum_{i=1}^n \pair{v, E_i(x)} E_i(x), \label{eqn:P}
\end{align}
for $v\in\Re^{\bar n}$. 
While \eqref{eqn:P} is written in terms of a specific orthonormal frame, the orthonormal projection is uniquely determined once the inner product is chosen. 
The projection operator is linear, and when represented by an $\bar n \times \bar n$ matrix, it satisfies the following properties 
\begin{gather}
    P(x)^2 = P(x) \quad \text{(idempotent)},\label{eqn:P^2}\\
    P^T(x) = P(x) \quad \text{(symmetric, self-adjoint)}.\label{eqn:P_sym}
\end{gather}
When $P(x)$ is not represented by a matrix, the left-hand side of \eqref{eqn:P} is also written as $(P(x))(v)$ to denote the projection of $v\in\Re^{\bar n}$ to $T_x M$. 

Further, it decomposes the ambient Euclidean space $\Re^{\bar n}$ into the $n$-dimensional tangent space $T_x M = \mathrm{range}[ P(x) ]$ and its orthogonal complement $(T_x M)^\perp = \mathrm{kernel}[ P(x) ]$, referred to as the \textit{normal space}, whose dimension is $\bar n - n$. 
As before, the vector field on $M$ is denoted by $\mathfrak{X}(M)$, and the vector field to the normal space is denoted by $\mathfrak{X}^\perp(M)$, i.e., given $N\in\mathfrak{X}^\perp(M)$, we have $N(x)\in (T_x M)^\perp$ for any $x\in M$.

A smooth function $f\in C^\infty(M)$ on the manifold can be \textit{extended} to $\bar f \in C^\infty(\Re^{\bar n})$ on the ambient $\Re^{\bar n}$, by the condition that
\begin{align}
    f(x) = \bar f(x), \label{eqn:f_bar}
\end{align}
for any $x\in M$. 
The formulation of $\bar f$ is not unique, as it can take any value outside of $M$ as long as it is smooth. 
Similarly, a vector field $X\in\mathfrak{X}(M)$ is extended to $\bar X\in\mathfrak{X}(\Re^{\bar n})$ with
\begin{align}
    X(x) = \bar X(x), \label{eqn:X_bar}
\end{align}
for any $x\in M$. 
In the subsequent developments, the over line indicates an extension of a function or a vector field to $\Re^{\bar n}$, or the standard operator on $\Re^{\bar n}$. 
For example, as shown in \eqref{eqn:f_bar} and \eqref{eqn:X_bar}, $\bar f$ and $\bar X$ denote extension of $f\in C^\infty(M)$ and $X\in\mathfrak{X}(M)$, respectively. 
For another example, $\grad$ denotes the gradient on $M$ defined by \eqref{eqn:grad}, and $\overline{\grad}$ represents the gradient operator on the ambient $\Re^{\bar n}$, which is the standard gradient in the classic calculus. 
These extensions allow the standard operations on $\Re^{\bar n}$ be utilized to obtain the corresponding operators on $M$, which are independent of the specific choice of the extension. 

First, it has been shown that the covariant derivative on $T_xM$ is related to that on $\Re^{\bar n}$ as follows.
\begin{theorem}{(Gauss Formula \cite{lee2006riemannian})}\label{thm:Gauss}
Let $X,Y\in\mathfrak{X}(M)$, and suppose $\bar X, \bar Y \in\mathfrak{X}(\Re^{\bar n})$ are their extension. 
Then,
    \begin{align}
        \overline\nabla_{\bar X} \bar Y = \nabla_X Y + \mathrm{II}(X,Y), \label{eqn:Gauss_formula}
    \end{align}
    where $\mathrm{II}:\mathfrak{X}(M)\times\mathfrak{X}(M)\rightarrow\mathfrak{X}^\perp(M)$ is symmetric and bilinear, and it is referred to as the \textit{second fundamental form}.\footnote{The first fundamental form corresponds to the Riemannian metric.}
\end{theorem}

This theorem states that when the covariant derivative on the ambient Euclidean space is decomposed into the tangent space $T_x M$ and the normal space $(T_xM)^\perp$, the tangential component coincides with the covariant derivative on the embedded $M$. 
Equivalently, the difference between the covariant derivative $\bar\nabla_{\bar X}\bar Y$ in the ambient $\Re^{\bar n}$ and the covariant derivative $\nabla_X Y$ in the embedded $M$ is captured by the second fundamental form $\mathrm{II}(X,Y)$, which is normal to the tangent space. 

The second fundamental form encodes the \textit{extrinsic curvature}, or how much the manifold is curved when observed from the ambient Euclidean space in which the manifold is embedded.
Specifically, it defines the \textit{mean curvature vector}, $H\in\mathfrak{X}^\perp (M)$ by
\begin{align}
    H = \sum_{i=1}^{n} \mathrm{II}(E_i, E_i),  \label{eqn:mean_curvature}
\end{align}
where $\{E_i\}$ is the orthonormal frame~\cite[Chapter 6]{do1992riemannian}.\footnote{The definition in~\cite{do1992riemannian} includes an additional scaling factor $\frac{1}{n}$.}
As \eqref{eqn:mean_curvature} is essentially a trace operator, it is independent of the choice of the frame. 
However, it should be noted that both the second fundamental form and the mean curvature vector are extrinsic properties, and they depend on the choice of the ambient space and the specific way the manifold is embedded within it.

For a hypersurface with $\bar n = n+1$, the mean curvature vector reduces to the product of a unit-normal vector and a scalar-valued function on $M$, referred to as \textit{mean curvature}, which measures how the surface area changes if the surface is deformed at a point of $M$~\cite[Chapter 8]{lee2006riemannian}. 
As such, if the mean curvature vanishes at every point of $M$, it is referred to as \textit{minimal} (or more precisely critical) surface, e.g., plane or catenoid.

Next, we find how several operators on the manifold are related to their extensions on the ambient Euclidean space.
\begin{prop}
    Let $f\in C^\infty(M)$ and $X,Y\in \mathfrak{X}(M)$. 
    Their extensions to the ambient $\Re^{\bar n}$ are denoted by the over line. 
    The Lie derivative, covariant derivative, and Hessian on $M$ can be obtained by the corresponding operator in the ambient $\Re^{\bar n}$ and the projection operator, as given by
    \begin{align}
        X[f] & = \bar X[\bar f], \label{eqn:Lie_deriv_proj}\\
        \nabla_X Y & = P(x) (\bar\nabla_{\bar X} \bar Y), \label{eqn:cov_deriv_proj}\\
        \mathrm{Hess}_f (X,Y) & = \overline{\mathrm{Hess}}_{\bar f} (\bar X, \bar Y) + (\mathrm{II}(X,Y))[f], \label{eqn:Hess_proj}
    \end{align}
    where $\mathrm{II}$ denotes the second fundamental form introduced in \Cref{thm:Gauss}.
\end{prop}
\begin{proof}
    Let $\gamma:\Re\rightarrow M$ be a path on $M$, satisfying $\gamma(0)=x$ and $\gamma'(0)=X(x)$. 
    We have
    \begin{align*}
        X[f] = \frac{d}{d\epsilon}\bigg|_{\epsilon=0} f(\gamma(\epsilon)).
    \end{align*}
    But, this is identical to
    \begin{align*}
        \bar X[\bar f] = \frac{d}{d\epsilon}\bigg|_{\epsilon=0} \bar f(\gamma(\epsilon)),
    \end{align*}
    as $f\circ\gamma = \bar f\circ\gamma $ and $\gamma'(0)=X(\gamma(0))=\bar X(\gamma(0))$. 
    This shows \eqref{eqn:Lie_deriv_proj}.

    Next, \eqref{eqn:cov_deriv_proj} is a direct consequence of \eqref{eqn:Gauss_formula}.
    Alternatively, this can be verified by showing that the right-hand side of \eqref{eqn:cov_deriv_proj} satisfies the defining properties of the covariant derivative on $M$.

    Finally, from \eqref{eqn:Hess} and \eqref{eqn:Lie_deriv_proj}, the Hessian on $M$ is written as
    \begin{align*}
        \mathrm{Hess}_f(X,Y) & = \bar X[\bar Y[\bar f]] - (\nabla_X Y)[f]\\
                             & = \overline{\mathrm{Hess}}_{\bar f}(\bar X, \bar Y) + (\bar\nabla_{\bar X}\bar Y)[f] -(\nabla_X Y)[f],
    \end{align*}
    where the second equality is obtained by applying \eqref{eqn:Hess} in the ambient $\Re^{\bar n}$.
    This yields \eqref{eqn:Hess_proj} when combined with the Gauss formula \eqref{eqn:Gauss_formula}.
\end{proof}

\subsection{Brownian Motion on an Embedded Manifold}

Now, we construct a stochastic differential equation on the ambient Euclidean space $\Re^{\bar n}$ to obtain Brownian motion on the embedded manifold $M$. 
In \Cref{sec:BM_Ei}, we injected the noise along each axis of an orthonormal frame, and designed the drift such that the resulting generator is half of the Laplace--Beltrami operator. 
While this approach is mathematically sound, it often yields a local formulation as it is impossible to choose an orthonormal frame that is globally valid on $M$ in general.

Here, we bypass this issue by projecting the orthonormal frame of the ambient $\Re^{\bar n}$ into the tangent space $T_x M$ as follows.
\begin{definition}[Pseudo-frame]\label{def:pseudo-frame}
    Let $\{e_1,\ldots, e_{\bar n}\}$ be an orthonormal basis in the ambient Euclidean space $\Re^{\bar n}$. 
The \textit{pseudo-frame} is a set of tangent vectors in $T_x M$ obtained by
\begin{align}
    \{ P(x)e_1, P(x) e_2, \ldots P(x) e_{\bar n}\}\in (T_x M)^{\bar n}, \label{eqn:pseudo_frame}
\end{align}
where $P$ is the orthonormal projection presented in \eqref{eqn:P}.
\end{definition}
Clearly, this set spans the tangent space $T_x M$, as its range is the range of the projection operator $P(x)$ at $x\in M$. 
However, it is neither linearly independent nor orthonormal. 
For example, there is a redundancy as the $n$-dimensional tangent space is spanned by $\bar n > n$ vectors, and  $\pair{P(x)e_i, P(x)e_j}\neq \delta_{ij}$ in general.

However, the desirable feature is that \eqref{eqn:pseudo_frame} is constructive in a global fashion, thereby avoiding the limitation of the local formulation of orthonormal frames on $M$.
Further, \eqref{eqn:pseudo_frame} can be utilized to compute the trace of a linear operator on $T_x M$ as summarized below.

\begin{prop}
    Let $A:T_xM\rightarrow T_xM$ be a linear operator on $T_xM$, and let $\bar A$ be its extension to $\Re^{\bar n}$. 
    The trace of $A$ at $x\in M$ is obtained by using the pseudo-frame \eqref{eqn:pseudo_frame} as
    \begin{align}
        \sum_{i=1}^n \pair{E_i(x), A E_i(x)} = \sum_{i=1}^{\bar n} \pair{P(x) e_i, \bar A P(x)e_i}, \label{eqn:tr_proj}
    \end{align}
    where $\{E_i\}$ is an orthonormal frame of $M$. 
\end{prop}
\begin{proof}
    Let an orthonormal basis of $T_x M$ be $\{t_1, \ldots, t_n\}$, and let an orthonormal basis of $(T_x M)^\perp$ be $\{n_{n+1},\ldots, n_{\bar n}\}$. 
    Together, they form an orthonormal basis of $T_x\Re^{\bar n} = T_x M \oplus (T_x M)^\perp$.

    Using $P=P^T$, the right-hand side of \eqref{eqn:tr_proj} is rearranged into
    \begin{align*}
        \sum_{i=1}^{\bar n} \pair{P(x) e_i, \bar A P(x)e_i} = \sum_{i=1}^{\bar n} \pair{e_i, (P(x) \bar A P(x))e_i},
    \end{align*}
    which is interpreted as the trace of a linear operator $P\bar AP$. 
    Since the trace is independent of the choice of the orthonormal frame, we utilize $\{t_i\}\cup\{n_i\}$ to obtain 
    \begin{align*}
        \sum_{i=1}^{\bar n}  \langle P(x) e_i, \bar A P(x)e_i \rangle & = \sum_{i=1}^{n} \pair{t_i, (P(x) \bar A P(x))t_i}\\
                                                                      & + \sum_{i=n+1}^{\bar n} \pair{n_i, (P(x) \bar A P(x))n_i}.
    \end{align*}
    Here, the second term of the right-hand side vanishes, as $P(x)n_i=0$, 
    and using $P=P^T$ and $Pt_i = t_i$, the first term can be rearranged to obtain
    \begin{align*}
        \sum_{i=1}^{\bar n} \pair{P(x) e_i, \bar A P(x)e_i} = \sum_{i=1}^{n} \pair{t_i, A t_i},
    \end{align*}
    which corresponds to the trace of $A$ on $M$. 
    Again, as the trace is independent of the choice of the orthonormal basis, the above is identical to the left-hand side of \eqref{eqn:tr_proj}.
\end{proof}

Equation \eqref{eqn:tr_proj} establishes that the intrinsic trace of a linear operator on the tangent space $T_xM$ can be computed by summing over the pseudo-frame given by \eqref{eqn:pseudo_frame}.
Now, utilizing this, we construct a Brownian motion by injecting noise along the vectors of the pseudo-frame \eqref{eqn:pseudo_frame}.

\begin{theorem}[Brownian Motion on an Embedded Manifold]\label{thm:Brownian_motion_embedded}
    Let $M$ be an $n$-dimensional manifold embedded in $\Re^{\bar n}$ for $\bar n > n$. 
    Additionally, let $\{e_i\}$ be the standard basis of $\Re^{\bar n}$ and let $P(x):\Re^{\bar n}\rightarrow T_x M$ be the orthogonal projection given by \eqref{eqn:P}.
    Consider the following Stratonovich stochastic differential equation
    \begin{align}
        dx & = -\frac{1}{2} \sum_{i=1}^{\bar n} \nabla_{P(x)e_i} P(x)e_i\, dt + \sum_{i=1}^{\bar n} P(x)e_i \circ dW_i\label{eqn:SDE_Bro_Emb},
    \end{align}
    and its equivalent It\^{o} form
    \begin{align}
        dx & = \frac{1}{2} H\, dt + \sum_{i=1}^{\bar n} P(x)e_i\, dW_i. \label{eqn:SDE_Bro_Emb_Ito}
    \end{align}
    Then the infinitesimal generator of either \eqref{eqn:SDE_Bro_Emb} or \eqref{eqn:SDE_Bro_Emb_Ito} is 
    \[
        \mathcal{A} f = \frac12 \Delta f,
    \]
    and the resulting process is the Brownian motion on $M$.
\end{theorem}
\begin{proof}
    From \eqref{eqn:A_Ito_Hess}, the generator of \eqref{eqn:SDE_Bro_Emb_Ito} is given by
    \begin{align}
        \mathcal{A} f = \frac{1}{2} H[f] + \frac{1}{2} \sum_{i=1}^{\bar n} \overline{\mathrm{Hess}}_{\bar f} (Pe_i, Pe_i), \label{eqn:A_emb_0}
    \end{align}
    where the over line denotes the operators defined in the ambient $\Re^{\bar n}$. 
    As the Hessian is bilinear and symmetric, there is a linear operator $\Hessmatrix: T_x\Re^{\bar n} \rightarrow T_x\Re^{\bar n}$ defined such that
    \begin{align}
        \overline{\mathrm{Hess}}_{\bar f} (X, Y) = \pair{  \Hessmatrix X , Y}. \label{eqn:Hess_matrix}
    \end{align}
    Thus, we have
    \begin{align}
        \sum_{i=1}^{\bar n} &\, \overline{\mathrm{Hess}}_{\bar f} (Pe_i, P e_i) 
        = \sum_{i=1}^{\bar n} \pair{ P e_i, \Hessmatrix Pe_i} \nonumber \\
                            & = \sum_{i=1}^n \pair{ E_i, \Hessmatrix E_i}\quad \text{(from \eqref{eqn:tr_proj})}\nonumber\\
                            & = \sum_{i=1}^{n} \overline{\mathrm{Hess}}_{\bar f} (E_i, E_i).\quad \text{(from \eqref{eqn:Hess_matrix})}  \label{eqn:Hess_tmp_0}
    \end{align}
    Note that the Hessian in the last equality is the usual Hessian in $\Re^{\bar n}$. 
    Therefore, it can be converted into the intrinsic Hessian on $M$ by \eqref{eqn:Hess_proj} to obtain
    \begin{align*}
        \sum_{i=1}^{\bar n} \overline{\mathrm{Hess}}_{\bar f} (Pe_i, P e_i)
                            & = \sum_{i=1}^{n} \left(\mathrm{Hess}_{f} (E_i, E_i) - (\mathrm{II}(E_i, E_i))[f]\right).
    \end{align*}
    Substituting this into \eqref{eqn:A_emb_0}, the drift term of \eqref{eqn:A_emb_0} is canceled out as $H = \sum_{i=1}^n \mathrm{II} (E_i, E_i)$ from \eqref{eqn:mean_curvature}. 
    Therefore, the generator reduces to
    \begin{align*}
        \mathcal{A} f= \frac{1}{2} \sum_{i=1}^{n} \mathrm{Hess}_{f} (E_i, E_i) = \frac{1}{2}\Delta f,
    \end{align*}
    where the second equality is due to \eqref{eqn:laplace_Ei_Hess}.
    Thus, the generator of \eqref{eqn:SDE_Bro_Emb_Ito} is half of the Laplace--Beltrami operator, and it represents the Brownian motion on $M$.

    Next, we convert \eqref{eqn:SDE_Bro_Emb_Ito} into a Stratonovich form. 
    Using the Gauss formula \eqref{eqn:Gauss_formula}, the It\^{o}--Stratonovich conversion term in \eqref{eqn:X_tilde} is
    \begin{align*}
        \frac{1}{2}\sum_{i=1}^{\bar n} \bar \nabla_{\overline{Pe_i}} \overline{Pe_i} = 
        \frac{1}{2} \sum_{i=1}^{\bar n} \{\nabla_{Pe_i} Pe_i +  \mathrm{II}(P e_i, P e_i)\}.
    \end{align*}
    Using the fact that the second fundamental form $\mathrm{II}$ is symmetric and bilinear, we can follow the same development as presented in \eqref{eqn:Hess_tmp_0} for the last term of the above to obtain
    \begin{align*}
        \frac{1}{2}\sum_{i=1}^{\bar n} \bar \nabla_{\overline{Pe_i}} \overline{Pe_i} 
        & = \frac{1}{2} \sum_{i=1}^{\bar n} \nabla_{Pe_i} Pe_i + \frac{1}{2}\sum_{i=1}^n \mathrm{II}(E_i, E_i) \\
        & = \frac{1}{2} \sum_{i=1}^{\bar n} \nabla_{Pe_i} Pe_i + \frac{1}{2} H,
    \end{align*}
    where the second equality follows from \eqref{eqn:mean_curvature}.
    Substituting this into \eqref{eqn:X_tilde} with $\tilde X = \frac{1}{2}  H$, the drift field of the Stratonovich form is
    \begin{align*}
        X = \tilde X - \frac{1}{2}\sum_{i=1}^{\bar n} \bar \nabla_{\overline{Pe_i}} \overline{Pe_i} =
-\frac{1}{2} \sum_{i=1}^{\bar n} \nabla_{Pe_i} Pe_i.
    \end{align*}
    This shows \eqref{eqn:SDE_Bro_Emb}.
\end{proof}

Compared with \Cref{thm:Brownian_motion}, the formulation of stochastic differential equations on the ambient Euclidean space avoids the challenges and limitations of utilizing a local orthogonal frame on $M$.
One drawback is that it requires $\bar n>n$ Wiener processes, when constructing a Brownian motion on an $n$-dimensional manifold $M$. 

It is interesting to observe that the Stratonovich form \eqref{eqn:SDE_Bro_Emb} in $\Re^{\bar n}$ has the same structure as \eqref{eqn:SDE_Brownian} on $M$.
Specifically, the orthonormal frame $\{E_i\}$ in \eqref{eqn:SDE_Brownian} is replaced by the pseudo-frame $\{Pe_i\}$ in \eqref{eqn:SDE_Bro_Emb}.
However, the It\^{o} formulation \eqref{eqn:SDE_Bro_Emb_Ito} in $\Re^{\bar n}$ requires an additional drift vector field obtained by the mean curvature vector, namely $\frac{1}{2} H$, compared to \eqref{eqn:SDE_Brownian_Ito}.

\begin{remark}
It is intriguing to observe that the drift vector field in the It\^{o} formulation \eqref{eqn:SDE_Bro_Emb_Ito}, which is required to ensure that the sample path evolves on the manifold, is precisely the mean curvature vector,
which measures the extrinsic curvature of the manifold, or the degree to which the manifold is bent in the ambient space. Importantly, this drift term is purely normal to the tangent space, hence it does not affect the generator acting on intrinsic functions.

For example, let $\gamma(t)$ be a geodesic on a manifold passing through a point $p\in M$ at $t=0$, i.e., $\gamma(0)=p$. As a geodesic, it is acceleration-free on the manifold, $\nabla_{\dot\gamma}\dot\gamma = 0$. However, its second fundamental form $\mathrm{II}(\dot\gamma(0), \dot\gamma(0))$ captures the normal acceleration at $p$ in the ambient Euclidean space, ensuring that the geodesic curve remains on the manifold. 

While a Brownian sample path is not smooth, the mean curvature vector provides the analogous normal acceleration in the ambient Euclidean space that keeps the process confined to the manifold. Because Brownian motion is isotropic, all tangent directions are equally likely, and thus the total correction is obtained by summing the second fundamental form $\mathrm{II}(E_i,E_i)$ over an orthonormal frame. 

This exemplifies how a fundamental geometric quantity of the manifold plays a critical role in shaping stochastic processes on manifolds.
This provides a unique geometric interpretations for Brownian motion on a manifold. 
\end{remark}

\subsection{Brownian Motion on an Embedded Lie Group}

Next, we discuss how the results of \Cref{thm:Brownian_motion_G} and \Cref{thm:Brownian_motion_embedded} can be applied to a Lie group $G$ embedded in a Euclidean space $\mathbb{R}^{\bar n}$. 
In \Cref{thm:Brownian_motion_G}, Brownian motion on a Lie group is obtained using an orthonormal frame constructed from the left-invariant metric defined by \eqref{eqn:metric_G}, 
whereas in \Cref{thm:Brownian_motion_embedded}, Brownian motion on an embedded manifold is formulated based on the metric induced from the ambient Euclidean metric through the embedding. 

These two formulations can be unified when the two metrics coincide, that is, when the induced metric is left-invariant:
\begin{align}
    \langle (L_g)_*\eta, (L_g)_*\zeta \rangle_{\mathbb{R}^{\bar n}}
    = \langle \eta, \zeta \rangle_{\mathbb{R}^{\bar n}}, 
    \label{eqn:metric_compatible}
\end{align}
for any $\eta, \zeta \in \mathfrak{g}$.
Here the subscript $\mathbb{R}^{\bar n}$ explicitly indicates that 
$\langle \cdot, \cdot \rangle_{\mathbb{R}^{\bar n}}$ denotes the metric in the ambient Euclidean space.
When \eqref{eqn:metric_compatible} is satisfied, Brownian motion on the embedded Lie group can be constructed consistently from either formulation, leading to the following unified construction.

\begin{theorem}[Brownian Motion on an Embedded Lie Group]\label{thm:Brownian_motion_G_Emb}
    Let $G$ be a Lie group embedded in the Euclidean space, whose induced metric is left-invariant, i.e., \eqref{eqn:metric_compatible} is satisfied.
    Let $\{e^\liealgebra_1,\dots,e^\liealgebra_n\}$ be an orthonormal basis on $\liealgebra$, constructed by \eqref{eqn:basis_algebra}, and
    we identify $\liealgebra^*$ with $\liealgebra$ using the inner product on $\liealgebra$.

    Consider the following Stratonovich stochastic differential equation
    \begin{align}
        g^{-1}dg &= \frac12 \sum_{i=1}^n \ad^*_{e_i^\liealgebra} e_i^\liealgebra \, dt + \sum_{i=1}^n e_i^\liealgebra \circ dW_i,
        \label{eqn:SDE_Brownian_G_Emb}
    \end{align}
    and its equivalent It\^{o} form
    \begin{align}
        dg &=  \frac{1}{2} H \,dt + \sum_{i=1}^n g e_i^\liealgebra  dW_i.
        \label{eqn:SDE_Brownian_Ito_G_Emb}
    \end{align}
    Then the infinitesimal generator of either \eqref{eqn:SDE_Brownian_G_Emb} or \eqref{eqn:SDE_Brownian_Ito_G_Emb} is 
    \[
        \mathcal{A} f = \frac12 \Delta f,
    \]
    such that the resulting process is the Brownian motion on $G$.
\end{theorem}
\begin{proof}
    From \eqref{eqn:A_Ito_Hess}, the generator of \eqref{eqn:SDE_Brownian_Ito_G_Emb} is given by
    \begin{align}
        \mathcal{A} f = \frac{1}{2} H[f] + \frac{1}{2} \sum_{i=1}^{n} \overline{\mathrm{Hess}}_{\bar f} (ge_i^\liealgebra, ge_i^\liealgebra),
    \end{align}
    where the Hessian in the last term can be converted into the intrinsic Hessian on $M$ by \eqref{eqn:Hess_proj}.
    And substituting the definition of the mean curvature vector \eqref{eqn:mean_curvature}, this reduces to
    \begin{align*}
        \mathcal{A} f= \frac{1}{2} \sum_{i=1}^{n} \mathrm{Hess}_{f} (E_i, E_i) = \frac{1}{2}\Delta f.
    \end{align*}
    Thus, \eqref{eqn:SDE_Brownian_Ito_G_Emb} represents the Brownian motion on $G$.

    Next, we convert \eqref{eqn:SDE_Brownian_Ito_G_Emb} into a Stratonovich form. 
    Using the Gauss formula \eqref{eqn:Gauss_formula} and \eqref{eqn:mean_curvature}, the It\^{o}--Stratonovich conversion term in \eqref{eqn:X_tilde} is
    \begin{align*}
        \frac{1}{2}\sum_{i=1}^{n} \bar \nabla_{ge_i^\liealgebra} ge_i^\liealgebra = 
        \frac{1}{2} \sum_{i=1}^{n} \{\nabla_{ge_i^\liealgebra } g e_i^\liealgebra +  \mathrm{II}(g e_i^\liealgebra, g e_i^\liealgebra)\}.
    \end{align*}
    Substituting \eqref{eqn:nabla_Ei_Ei_G} and \eqref{eqn:mean_curvature}, this reduces to
    \begin{align*}
        \frac{1}{2}\sum_{i=1}^{n} \bar \nabla_{ge_i^\liealgebra} ge_i^\liealgebra = 
        - \frac{1}{2} \parenth{\sum_{i=1}^{n} g \ad^*_{e_i^\liealgebra} e_i^\liealgebra}  +  \frac12 H.
    \end{align*}
    Substituting this into \eqref{eqn:X_tilde} with $\tilde X = \frac{1}{2}  H$, the drift field of the Stratonovich form is
    \begin{align*}
        X = \tilde X - \frac{1}{2}\sum_{i=1}^{n} \bar \nabla_{ge_i^\liealgebra} ge_i^\liealgebra =
        \frac{1}{2} \parenth{\sum_{i=1}^{n} g \ad^*_{e_i^\liealgebra} e_i^\liealgebra},
    \end{align*}
    which shows \eqref{eqn:SDE_Brownian_G_Emb}.
\end{proof}

Comparing this result with \Cref{thm:Brownian_motion_G} for Brownian motion on $G$, the Stratonovich form \eqref{eqn:SDE_Brownian_G_Emb} is identical to \eqref{eqn:SDE_Brownian_G}.
The It\^{o} form \eqref{eqn:SDE_Brownian_Ito_G_Emb} includes the additional drift field given by the mean curvature vector when compared with \eqref{eqn:SDE_Brownian_Ito_G}.

As presented in \Cref{cor:Brownian_motion_G_unimodular}, the drift term of \eqref{eqn:SDE_Brownian_G_Emb} vanishes when the Lie group isunimodular. 
Further, the expression of the mean curvature vector is simplified as follows. 

\begin{corollary}[Brownian Motion on an Embedded Unimodular Group]\label{cor:Brownian_motion_G_unimodular_Emb}
    If the embedded Lie group $G$ satisfying \eqref{eqn:metric_compatible} is unimodular, the stochastic differential equations \eqref{eqn:SDE_Brownian_G_Emb} and \eqref{eqn:SDE_Brownian_Ito_G_Emb} for Brownian motion reduce to
    \begin{gather}
        g^{-1}dg = \sum_{i=1}^n e_i^\liealgebra \circ dW_i, \label{eqn:SDE_Brownian_G_unimodular_Emb}\\
        dg = \frac{1}{2} H dt + \sum_{i=1}^n ge_i^\liealgebra dW_i, \label{eqn:SDE_Brownian_Ito_G_unimodular_Emb}
    \end{gather}
    where the mean curvature vector $H\in (T_gG)^\perp$ is obtained by
    \begin{align}
        H = \sum_{i=1}^n \bar\nabla_{\bar E_i} \bar E_i, \label{eqn:H_G_unimodular}
    \end{align}
    where $\bar \nabla$ denotes the covariant derivative in the ambient Euclidean space $\Re^{\bar n}$, and $\bar E_i\in\mathfrak{X}(\Re^{\bar n})$ is the extension of $E_i=ge_i^\liealgebra$ to $\Re^{\bar n}$. 
\end{corollary}
\begin{proof}
    In the proof of \Cref{cor:Brownian_motion_G_unimodular}, we showed that
    \begin{align*}
        J = \sum_{i=1}^n \ad^*_{e_i^\liealgebra} e_i^\liealgebra = 0,
    \end{align*}
    when $G$ is unimodular.
    Substituting this into \eqref{eqn:SDE_Brownian_G_Emb} yields \eqref{eqn:SDE_Brownian_G_unimodular_Emb}.

    Next, substituting \eqref{eqn:Gauss_formula} into \eqref{eqn:mean_curvature},
    \begin{align*}
        H = \sum_{i=1}^n (\bar\nabla_{\bar E_i} \bar E_i - \nabla_{E_i} E_i ) 
        = \sum_{i=1}^n \bar\nabla_{\bar E_i} \bar E_i + g J
    \end{align*}
    where the second equality is from \eqref{eqn:nabla_Ei_Ei_G}. 
    But, since $J=0$, this is further reduced to \eqref{eqn:H_G_unimodular}.
    Substituting this into \eqref{eqn:SDE_Brownian_Ito_G_Emb}, we obtain \eqref{eqn:SDE_Brownian_Ito_G_unimodular_Emb}.
\end{proof}

\begin{remark}
    \Cref{thm:Brownian_motion_G_Emb} and \Cref{cor:Brownian_motion_G_unimodular_Emb} apply to an embedded Lie group satisfying \eqref{eqn:metric_compatible}, or a Lie group whose the left-translation acts as an isometry of the Euclidean space. 
    This includes $O(n)$, $U(n)$, $SO(n)$, or $SU(n)$ (also $Sp(n)$ if the Euclidean metric is weighted by the canonical symplectic matrix). 
    For other embedded Lie groups, the intrinsic formulation of the Stratonovich form given by \eqref{eqn:SDE_Brownian_G} in \Cref{thm:Brownian_motion_G} (or \eqref{eqn:SDE_Brownian_G_unimodular} in \Cref{cor:Brownian_motion_G_unimodular} for a unimodular group) holds (see \Cref{sec:Torus_G}). 
    Also, considering the group as an embedded Riemannian manifold with the induced metric, \Cref{thm:Brownian_motion_embedded} can be utilized as well. 
\end{remark}

\section{Examples}\label{sec:ex}

In this section, we formulate Brownian motion for several manifolds and Lie groups listed below. 

\begin{center}
\begin{tblr}{
  width = \linewidth, 
  colspec = {X[1.3, m,c] X[1.1,m,c] X[m,c]}, 
  rowspec = {|[1.0pt]Q|[0.5pt]QQQQQQQ|[1.0pt]}
}
    Name & Type & Thm./Cor. \\
    Two-sphere (\S\ref{sec:Two-Sphere}) & Manifold & \Cref{thm:Brownian_motion} \\
    Torus (\S\ref{sec:Torus}) & Manifold & \Cref{thm:Brownian_motion} \\
    Torus (\S\ref{sec:Torus_G}) & Unimodular Group & \Cref{cor:Brownian_motion_G_unimodular} \\
    Hyperbolic Space (\S\ref{sec:Hyper}) & Manifold & \Cref{thm:Brownian_motion} \\
    $n$-Sphere (\S\ref{sec:n-Sphere}) & Embedded Manifold & \Cref{thm:Brownian_motion_embedded} \\
    Special Orthogonal Group (\S\ref{sec:SO3}) & Embedded Unimodular Group & \Cref{cor:Brownian_motion_G_unimodular_Emb} \\
    Affine Group (\S\ref{sec:Affine}) & Non-unimodular Group & \Cref{thm:Brownian_motion_G}
\end{tblr}
\end{center}

\subsection{Two-Sphere}\label{sec:Two-Sphere}

Consider a stochastic process on the two-sphere $\Sph^2=\{x\in\Re^3\,|\, \|x\|=1\}\subset \Re^3$. 
We develop Brownian motion using \Cref{thm:Brownian_motion}, considering $\Sph^2$ as a Riemannian manifold.  

\paragraph{Riemannian Metric}

A point $x\in\Sph^2$ is parameterized by
\begin{align*}
    x = \begin{bmatrix}\cos\phi \sin\theta & \sin\phi \sin\theta & \cos\theta \end{bmatrix},
\end{align*}
where $\theta\in[0,\pi]$ is the co-latitude, and $\phi \in [0,2\pi)$ is the longitude. 

The tangent vectors are given by
\begin{align*}
    \deriv{x}{\theta} & = \begin{bmatrix} \cos\phi \cos\theta & \sin\phi \cos\theta & -\sin\theta\end{bmatrix},\\
    \deriv{x}{\phi} & = \begin{bmatrix} -\sin\phi \sin\theta & \cos\phi \sin\theta & 0\end{bmatrix}.
\end{align*}
The resulting components of the metric tensor induced by the standard metric on $\Re^3$ are
\begin{align*}
    \g_{\theta\theta} & = \pair{\deriv{x}{\theta},\, \deriv{x}{\theta}}_{\Re^3} = 1,\\
    \g_{\theta\phi} & = \pair{\deriv{x}{\theta},\, \deriv{x}{\phi}}_{\Re^3} = 0,\\
    \g_{\phi\phi} & = \pair{\deriv{x}{\phi},\, \deriv{x}{\phi}}_{\Re^3} = \sin^2\theta.
\end{align*}
The Christoffel symbols are obtained by \eqref{eqn:Gamma} as
\begin{align*}
    \Gamma^\theta_{\phi\phi} = - \sin\theta\cos\theta,\quad
    \Gamma^\phi_{\theta\phi} = \Gamma^\phi_{\phi\theta} = \cot\theta,
\end{align*}
and all others vanish. 

\paragraph{Orthonormal Frame}

Let an orthonormal frame be
\begin{align*}
    E_\theta & = \deriv{x}{\theta} = \deriv{}{\theta},\\
    E_\phi & = \frac{1}{\sin\theta}\deriv{x}{\phi} = \frac{1}{\sin\theta} \deriv{}{\phi}.
\end{align*}
It is straightforward to verify $\g(E_i, E_j) = \delta_{ij}$.  
The covariant derivatives of the frame are given by
\begin{align*}
    \nabla_{E_\theta} E_\theta & = \nabla_{\partial_\theta} \partial_\theta = \Gamma^k_{\theta\theta} \partial_k = 0,\\
    \nabla_{E_\phi} E_\phi & = \nabla_{\frac{1}{\sin\theta} \partial_\phi} \frac{1}{\sin\theta}\partial_\phi
    = \frac{1}{\sin\theta} \nabla_{\partial_\phi} \frac{1}{\sin\theta} \partial_\phi \\
                           & = \frac{1}{\sin^2\theta} \Gamma^k_{\phi\phi} \partial_k = -\cot\theta E_\theta,
\end{align*}
where we have applied \eqref{eqn:cov_deriv_Leibniz} for the penultimate equality. 
From \eqref{eqn:laplace_Ei}, the Laplace--Beltrami operator is
\begin{align}
    \Delta f 
    & = \frac{\partial^2 f}{\partial \theta^2} + \frac{1}{\sin^2\theta} \frac{\partial^2 f}{\partial \phi^2}
    + \cot \theta \deriv{f}{\theta},\label{eqn:Delta_sph}
\end{align}
which recovers the well-known spherical Laplacian. 

\paragraph{SDE for Brownian Motion}

From \eqref{eqn:SDE_Brownian}, the Stratonovich SDE for the Brownian motion is given by 
\begin{align}
    d\theta & = \frac{1}{2}\cot\theta\, dt +  dW_\theta,\\
    d\phi & = \frac{1}{\sin\theta} \circ  dW_\phi,
\end{align}
which is transformed into the It\^{o} SDE by \eqref{eqn:SDE_Brownian_Ito} as
\begin{align}
    d\theta & = dW_\theta, \label{eqn:SDE_Brownian_Ito_Sph_theta}\\
    d\phi & = \frac{1}{\sin\theta} dW_\phi. \label{eqn:SDE_Brownian_Ito_Sph_phi}
\end{align}
From \Cref{thm:generator}, we can verify that the generator of the above stochastic differential equations is given by the Laplace--Beltrami operator on the two-sphere.

\subsection{Torus}\label{sec:Torus}

Consider a stochastic process on a torus $\mathbb{T}^2 = \Sph^1 \times \Sph^1$,
which is considered as a Riemannian manifold to apply the results of \Cref{thm:Brownian_motion}.

\paragraph{Riemannian Metric}
A point $x\in\mathbb{T}^2$ is parameterized by
\begin{align}
x(\theta, \phi)
= \begin{bmatrix}
(R + r\cos\theta)\cos\phi\\
(R + r\cos\theta)\sin\phi\\
r\sin\theta
\end{bmatrix}, \label{eqn:torus_embedding}
\end{align}
where $\theta,\phi\in[0,2\pi)$ and $R> r>0$ are the major radius and the minor radius.

The tangent vectors are
\begin{align*}
    \deriv{x}{\theta} = 
    \begin{bmatrix}
        - r\sin\theta\cos\phi\\
        - r\sin\theta\sin\phi\\
        \;\,r\cos\theta
    \end{bmatrix},
    \quad
    \deriv{x}{\phi} = 
    \begin{bmatrix}
        -(R+r\cos\theta)\sin\phi\\
        \;\,(R+r\cos\theta)\cos\phi\\
        0
    \end{bmatrix}.
\end{align*}
The resulting metric is
\begin{align}
    \g_{\theta\theta}=r^2,\quad
    \g_{\theta\phi}=\g_{\phi\theta}=0,\quad
    \g_{\phi\phi}=(R+r\cos\theta)^2.\label{eqn:metric_Torus_induced}
\end{align}
The Christoffel symbols are 
\begin{align*}
    \Gamma^{\theta}_{\phi\phi} & =\frac{(R+r\cos\theta)\sin\theta}{r},\\
    \Gamma^{\phi}_{\theta\phi} & =\Gamma^{\phi}_{\phi\theta} =-\frac{r\sin\theta}{R+r\cos\theta},
\end{align*}
and all other $\Gamma^k_{ij}$ vanish.

\paragraph{Orthonormal Frame}

An orthonormal frame is chosen as
\begin{align*}
    E_\theta & = \frac{1}{r}\deriv{x}{\theta},\\
    E_\phi & = \frac{1}{R+r\cos\theta}\deriv{x}{\phi},
\end{align*}
and their covariant derivatives are
\begin{align*}
    \nabla_{E_\theta}E_\theta &=0, \\
    \nabla_{E_\phi}E_\phi & =\frac{\sin\theta}{\,R+r\cos\theta\,}E_\theta.
\end{align*}

The Laplace--Beltrami operator is
\begin{align*}
    \Delta f
    = \frac{1}{r^{2}}\,\frac{\partial^{2} f}{\partial\theta^{2}}
    + \frac{1}{(R+r\cos\theta)^{2}}\,\frac{\partial^{2} f}{\partial\phi^{2}}
    - \frac{\sin\theta}{r(R+r\cos\theta)}\frac{\partial f}{\partial\theta}.
\end{align*}

\paragraph{SDE for Brownian Motion}

From \eqref{eqn:SDE_Brownian}, the Stratonovich SDE for the Brownian motion is given by 
\begin{align}
    d\theta & = - \frac{\sin\theta}{2r(R+r\cos\theta)} dt +  \frac{1}{r} dW_\theta,\label{eqn:SDE_Brownian_Torus_1}\\
    d\phi & = \frac{1}{R+ r\cos\theta} \circ  dW_\phi,
\end{align}
where the drift in $\theta$ arises solely from the term $-\frac{1}{2}\nabla_{E_\phi} E_\phi$.
This is transformed into the It\^{o} SDE by \eqref{eqn:SDE_Brownian_Ito} as
\begin{align}
    d\theta & = \frac{1}{r} dW_\theta,\\
    d\phi & = \frac{1}{R+ r\cos\theta} dW_\phi.\label{eqn:SDE_Brownian_Torus_4}
\end{align}

\subsection{Torus as a Lie Group}\label{sec:Torus_G}

Consider a stochastic process on the torus $\mathbb{T}^2 = \Sph^1 \times \Sph^1$. 
Here, the torus is viewed as a Lie group embedded in $\mathbb{R}^3$ through \eqref{eqn:torus_embedding}.
This serves as an example of a unimodular Lie group embedded in a Euclidean space; however, the metric compatibility condition \eqref{eqn:metric_compatible} is not satisfied. 
Thus, we utilize \Cref{cor:Brownian_motion_G_unimodular} to construct the corresponding Brownian motion.

\paragraph{Riemannian Metric}

The identity element is given by $e = x(0,0) = [R+r,\, 0,\, 0]^T \in \Re^3$.  
For $x(\theta_1,\phi_1), x(\theta_2,\phi_2) \in \mathbb{T}^2$, the group operation is defined as
\begin{align*}
    x(\theta_1, \phi_1)  x(\theta_2, \phi_2)
    = x(\theta_1 + \theta_2,\, \phi_1 + \phi_2),
\end{align*}
and the inverse element is $x^{-1}(\theta, \phi) = x(-\theta, -\phi)$. 

The tangent space at $x = e$, or the Lie algebra $\mathfrak{t}^2$, is spanned by
\begin{align*}
    \frac{\partial x}{\partial \theta}\bigg|_{x=e} = 
    \begin{bmatrix} 0\\ 0 \\ r \end{bmatrix}, \qquad
    \frac{\partial x}{\partial \phi}\bigg|_{x=e} =
    \begin{bmatrix} 0\\ R+r \\ 0 \end{bmatrix},
\end{align*}
so that
\[
    \mathfrak{t}^2 = \{ (0, \eta_2, \eta_3) \in \Re^3 \mid \eta_2, \eta_3 \in \Re \}.
\]
The inner product on $\mathfrak{t}^2$ is defined by
\begin{align*}
    \langle \eta, \zeta \rangle_{\mathfrak{t}^2} = \eta^T \zeta = \eta_2 \zeta_2 + \eta_3 \zeta_3,
\end{align*}
for any $\eta=(0,\eta_2,\eta_3)$ and $\zeta=(0,\zeta_2,\zeta_3)$ in $\mathfrak{t}^2$.  
An orthonormal basis of $\mathfrak{t}^2$ is chosen as
\begin{align*}
    e_\theta^{\liealgebra} = [0,0,1]^T, \qquad e_\phi^{\liealgebra} = [0,1,0]^T.
\end{align*}

Since $\mathbb{T}^2$ is not considered as a matrix group, the notational shorthand representing the pushforward by left translation with matrix multiplication in the last equality of \eqref{eqn:X_eta} does not apply, i.e., $(L_g)_*\eta \neq g\eta$.  
Instead, the pushforward of $\eta = (0,\eta_2,\eta_3) \in \mathfrak{t}^2$ by the left translation is computed by
\begin{align*}
    (L_{x(\theta, \phi)})_* \eta
    = \frac{d}{dt}\bigg|_{t=0} x\big(\theta+\theta'(t),\, \phi+\phi'(t)\big),
\end{align*}
where $\theta'(t)$ and $\phi'(t)$ are chosen such that $\theta'(0)=\phi'(0)=0$ and
\[
    \frac{d}{dt}\bigg|_{t=0} x(\theta'(t), \phi'(t)) = \eta.
\]
For example, we can choose
\begin{align*}
    \theta'(t) = \frac{\eta_3}{r}t, \quad \phi'(t) = \frac{\eta_2}{R+r}t.
\end{align*}
By expanding the derivative, we obtain
\begin{align*}
    (L_{x(\theta, \phi)})_* \eta
    = \frac{\eta_3}{r} \frac{\partial x(\theta,\phi)}{\partial \theta}
    + \frac{\eta_2}{R+r} \frac{\partial x(\theta,\phi)}{\partial \phi}.
\end{align*}

Using this result, the left-invariant metric is given by
\begin{align}
    \g(v,w) = \pair{ (L_{x^{-1}})_* v, (L_{x^{-1}})_* w}_{\mathfrak{t}^2}, \label{eqn:metric_Torus_left}
\end{align}
for $v,w\in T_x\mathbb{T}^2$, and the left-invariant orthonormal basis on $T_x\mathbb{T}^2$ is given by
\begin{align*}
    E_\theta & = (L_{x(\theta,\phi)})_* e_\theta^{\liealgebra}
          = \frac{1}{r} \frac{\partial x(\theta,\phi)}{\partial \theta},\\
    E_\phi & = (L_{x(\theta,\phi)})_* e_\phi^{\liealgebra}
          = \frac{1}{R+r} \frac{\partial x(\theta,\phi)}{\partial \phi}.
\end{align*}

We compare this left-invariant metric \eqref{eqn:metric_Torus_left} with the induced metric \eqref{eqn:metric_Torus_induced}.
Given that $\partial_\theta = r E_\theta$ and $\partial_\phi = (R+r) E_\phi$, we have
\begin{gather*}
    \g_{\theta\theta} = \g(rE_\theta, rE_\theta) = r^2, \quad
    \g_{\theta\phi} = 0, \quad
    \g_{\phi\phi} = (R+r)^2,
\end{gather*}
which is distinct from \eqref{eqn:metric_Torus_induced} unless $\theta=0$. 
As such, \eqref{eqn:metric_compatible} is not satisfied.

The corresponding Laplace--Beltrami operator is
\begin{align*}
    \Delta f= \frac{1}{r^2} \frac{\partial^2 f}{\partial \theta^2} + \frac{1}{(R+r)^2} \frac{\partial^2 f}{\partial \phi^2}.
\end{align*}

\paragraph{SDE for Brownian Motion}
From \eqref{eqn:SDE_Brownian_G_unimodular} of \Cref{cor:Brownian_motion_G_unimodular}, the stochastic differential equation in the Stratonovich form for Brownian motion on $\mathbb{T}^2$ is
\begin{align}
    d x(\theta,\phi) = \frac{1}{r} \frac{\partial x(\theta,\phi)}{\partial \theta}\circ dW_1 + \frac{1}{R+r} \frac{\partial x(\theta,\phi)}{\partial \phi}\circ dW_2. \label{eqn:SDE_Brownian_Torus_left}
\end{align}
An equivalent It\^{o} formulation follows from \eqref{eqn:SDE_Brownian_Ito_G}, interpreted according to the intrinsic geometry of the embedded torus in $\Re^3$.
It is important to note that \eqref{eqn:SDE_Brownian_Torus_left} describes Brownian motion with respect to the left-invariant metric \eqref{eqn:metric_Torus_left}, whereas \eqref{eqn:SDE_Brownian_Torus_1}--\eqref{eqn:SDE_Brownian_Torus_4} are constructed by the induced metric \eqref{eqn:metric_Torus_induced}.

\subsection{Hyperbolic Space: Upper Half--Space Model}\label{sec:Hyper}

We now develop the Brownian motion on the $n$-dimensional hyperbolic space, given by
\[
\Hyper^n = \{ (x_1,\dots,x_{n}) \in \Re^n \mid x_n>0 \},
\]
which is formulated by the standard upper half-space model.
This is endowed with the Riemannian metric
\begin{align}
    \g = \frac{1}{x_n^2} I_{n\times n}, \label{eqn:metric_Hn}
\end{align}
and we use the results of \Cref{thm:Brownian_motion}.

\paragraph{Orthonormal Frame}
At a point $x\in\Hyper^n$, the tangent space $T_x\Hyper^n$ is identified with $\Re^n$.
Thus, a global orthonormal frame is given by
\begin{align*}
    E_i & = x_n \frac{\partial}{\partial x_i},
\end{align*}
for $i\in \mathcal{I} \triangleq \{1,\ldots,n\}$.

The Christoffel symbols of the metric \eqref{eqn:metric_Hn} are
\begin{gather*}
    \Gamma^{i}_{i n} = \Gamma^{i}_{n i} = -\frac{1}{x_n},\quad 
    \Gamma^{n}_{ii} = \frac{1}{x_n}, \quad \text{for $i\in\mathcal{I}\setminus\{n\}$},\\
    \Gamma^{n}_{nn} = -\frac{1}{x_n},
\end{gather*}
and all others vanish.
From these, the covariant derivatives of the orthonormal frame are
\begin{gather*}
    \nabla_{E_i} E_i = E_n, \quad \text{for $i\in\mathcal{I}\setminus\{n\}$},\\ 
    \nabla_{E_n} E_n = 0.
\end{gather*}

\paragraph{Laplace--Beltrami Operator}
Using \eqref{eqn:laplace_Ei}, the Laplace--Beltrami operator on $\Hyper^n$ is
\begin{align}
    \Delta f
    &= x_n^2 \sum_{i=1}^{n} \frac{\partial^2 f}{\partial x_i^2}
       -(n-2) x_n \frac{\partial f}{\partial x_n}.
    \label{eqn:Laplace_Hn}
\end{align}

\paragraph{SDE for Brownian Motion}
By \Cref{thm:Brownian_motion}, the Stratonovich stochastic differential equation for Brownian motion on $\Hyper^n$ is
\begin{align}
    dx_i &= x_n \circ dW_i, \quad \text{for $i\in\mathcal{I}\setminus\{n\}$},\\
    dx_n &= -\frac12 (n-1)x_n dt + x_n\circ dW_n.
\end{align}
Equivalently, the It\^{o} formulation is
\begin{align}
    dx_i &= x_n dW_i,
\end{align}
for all $i\in\mathcal{I}$.

\subsection{$n$-Sphere}\label{sec:n-Sphere}

We develop Brownian motion on the $n$-sphere $\Sph^n=\{x\in\Re^{n+1}\mid x^Tx=1\}$ embedded in $\Re^{\bar n}$ with $\bar n = n+1$,
using \Cref{thm:Brownian_motion_embedded}.

\paragraph{Riemannian Metric}

The metric is induced from the standard Euclidean inner product in $\Re^{\bar n}$, and the orthonormal basis of $\Re^{\bar n}$ is chosen as the standard basis of $\Re^{\bar n}$. 
The tangent space at $x$ is $T_x \Sph^n=\{v\in\Re^{n+1}\mid x\cdot v=0\}$, and the orthogonal projection is
\begin{equation}
    P(x)=I_{\bar n \times \bar n}-x x^T, \label{eqn:proj_Sn}
\end{equation}
where $I_{\bar n\times \bar n} \in\Re^{\bar n\times \bar n}$ is the identity matrix. 
We have $P^2=P$, $P^T=P$, and $\mathrm{tr}[P(x)]=n$.

\paragraph{Mean Curvature}

For $X,Y\in \mathfrak{X}(\Sph^n)$, let their extensions be $\bar X,\bar Y\in\mathfrak{X}(\Re^{\bar n})$. 
From \eqref{eqn:cov_deriv_proj}, the covariant derivative is given by
\begin{equation}
    \nabla_X Y = P(x)(\bar\nabla_{\bar X}\bar Y)=\bar\nabla_{\bar X}\bar Y - x\,x^T\bar\nabla_{\bar X}\bar Y. \label{eqn:cov_driv_Sn}
\end{equation}
Since the covariant derivative $\bar \nabla_{\bar X} \bar Y $ in $\Re^{\bar n}$ is the usual directional derivative of $\bar Y$ along $\bar X$, we have
\begin{align}
    x^T \bar\nabla_{\bar X} \bar Y =  \sum_{i=1}^{\bar n} \parenth{\sum_{j=1}^{\bar n} \deriv{\bar Y_i}{x_j}\bar X_j} x_i
    = \sum_{i,j=1}^{\bar n} \deriv{\bar Y_i}{x_j} X_j x_i. \label{eqn:tmp0}
\end{align}
We simplify this using $Y(x)\cdot x = \bar Y(x)\cdot x = \sum_{i=1}^{\bar n} \bar Y_i x_i = 0$.
Taking its derivative with respect to $x_j$, 
\begin{align*}
    \sum_{i=1}^{\bar n} \left( \deriv{\bar Y_i}{x_j} x_i + \bar Y_i \deriv{x_i}{x_j} \right) =
    \sum_{i=1}^{\bar n} \deriv{\bar Y_i}{x_j} x_i + \bar Y_j = 0,
\end{align*}
which is substituted into \eqref{eqn:tmp0} to obtain 
\begin{align*}
    x^T \nabla_{\bar X} \bar Y = - \sum_{j=1}^{\bar n} \bar Y_j \bar X_j = -X\cdot Y.
\end{align*}
Thus, from \eqref{eqn:cov_driv_Sn},
\begin{align*}
    \nabla_X Y = \bar\nabla_{\bar X} \bar Y + x (X\cdot Y). 
\end{align*}
By Gauss’ formula, the second fundamental form on $\Sph^n$ is 
\[
\mathrm{II}(X,Y)=\bar\nabla_{\bar X}\bar Y-\nabla_X Y = -x(X\cdot Y).
\]
Taking the trace over an orthonormal basis as shown in \eqref{eqn:tr_proj}, and using $\mathrm{tr}[P(x)]=n$,
we obtain the mean curvature vector as
\begin{align*}
    H(x) & =\sum_{i=1}^n \mathrm{II}(E_i,E_i)
         = -x\sum_{i=1}^{\bar n}\langle P(x)e_i, P(x)e_i\rangle\\
         & = -x\,\mathrm{tr}(P(x)) = - nx.
\end{align*}

\paragraph{Drift term in \eqref{eqn:SDE_Bro_Emb}}

Next, we compute the drift term of the Stratonovich form, \eqref{eqn:SDE_Bro_Emb}.
Using \eqref{eqn:cov_deriv_proj} and \eqref{eqn:proj_Sn},
\begin{equation*}
    \nabla_{Pe_i}{P e_i} 
     = P \parenth{ \bar\nabla_{Pe_i} Pe_i} = - P \braces{\bar\nabla_{Pe_i}  (x_i x)}.
\end{equation*}
Applying the Leibniz rule \eqref{eqn:cov_deriv_Leibniz},
\begin{align}
    \nabla_{Pe_i}{P e_i}& = - P \braces{((Pe_i)[x_i])x + x_i \bar\nabla_{P e_i} x}.\label{eqn:tmp_Sn}
\end{align}
But,
\begin{align*}
    (Pe_i)[x_i] & = \sum_{j=1}^{\bar n} \deriv{x_i}{x_j} e_j^T Pe_i = e_i^T P e_i = 1 - x_i^2, \\
    \bar\nabla_{Pe_i} x & = \sum_{j=1}^{\bar n} \deriv{x}{x_j} e_j^T P e_i = \parenth{\sum_{j=1}^{\bar n}e_j e_j^T} Pe_i = P e_i.
\end{align*}
Substituting these back to \eqref{eqn:tmp_Sn}, and using $Px=0$ and $P^2=P$,
\begin{align*}
    \nabla_{Pe_i}{P e_i} & = - P ( (1-x_i^2) x + x_i Pe_i) = -x_i P e_i.
\end{align*}
Finally, taking the sum, 
\begin{equation*}
    \sum_{i=1}^{\bar n} \nabla_{Pe_i}{Pe_i} = -\sum_{i=1}^{\bar n} x_i P e_i = -P(x)x=0.
\end{equation*}
Therefore, the drift term of \eqref{eqn:SDE_Bro_Emb} vanishes. 

\paragraph{SDE for Brownian Motion}
By \Cref{thm:Brownian_motion_embedded}, stochastic differential equations for Brownian motion on $\Sph^n$ are given by
\begin{gather}
d x = \sum_{i=1}^{\bar n} P(x) e_i \circ dW_i,\\
d x = -\frac{n}{2}x dt + \sum_{i=1}^{\bar n} P(x) e_i dW_i. \label{eqn:SDE_Brownian_Sn_Ito}
\end{gather}

Using It\^{o}’s rule and \eqref{eqn:SDE_Brownian_Sn_Ito}, we can show that
\begin{align*}
    d(x^Tx)& =(dx)^Tx + x^T(dx) + (dx)^T(dx)\\
           & = \big(-n + \mathrm{tr}[P(x)]\big)dt = (-n+n)\,dt = 0,
\end{align*}
which verifies that the sample trajectories of \eqref{eqn:SDE_Brownian_Sn_Ito} remain on $\Sph^n$.

\subsection{Special Orthogonal Group}\label{sec:SO3}

Here, we develop Brownian motion on $\SO3=\{R\in\Re^{3\times 3}\,|\, R^TR=I_{3\times 3},\; \mathrm{det}[R]=1\}$.
Since $\SO3$ is a unimodular group embedded in $\Re^{3\times 3}$, whose left-invariant metric is identical to the induced metric, i.e., \eqref{eqn:metric_compatible} holds, we use \Cref{cor:Brownian_motion_G_unimodular_Emb}.

\paragraph{Riemannian Metric}

The Lie algebra is $\so3=\{S\in\Re^{3\times 3}\,|\, S^T = -S\}$ with the bracket $[\eta, \xi] = \eta\xi - \xi\eta$ for $\eta,\xi\in\so3$.
Let the hat map $\wedge:\Re^3\rightarrow \so3=\{S\in\Re^{3\times 3}\,|\, S^T = -S\}$ be defined such that $\hat x y = x \times y$ for all $x,y\in\Re^3$. 
Then, $\so3$ is isomorphic to $\Re^3$ with the vector cross product in $\Re^3$. 
The inverse of the hat map is denoted by the vee map $\vee:\so3\rightarrow\Re^3$.

For $\eta,\xi\in\so3$, the inner product on $\so3$ is chosen as
\begin{align}
    \pair{\eta, \xi}_\liealgebra = \frac{1}{2}\mathrm{tr}[\eta^T \xi] = (\eta^\vee)^T(\xi^\vee). \label{eqn:metric_R33}
\end{align}
The factor $\frac{1}{2}$ simplifies the subsequent development. 
Let $\{e_1, e_2, e_3\}$ be the standard basis of $\Re^3$.
Then, an orthonormal basis of $\so3$ can be chosen as
\begin{align*}
    \{ e_1^\liealgebra, e_2^\liealgebra, e_3^\liealgebra \} =\{\hat e_1, \hat e_2, \hat e_3\},
\end{align*}
where we can verify that $\pair{\hat e_i, \hat e_j}_\liealgebra=\delta_{ij}$ with respect to the metric given by \eqref{eqn:metric_R33}.

The left-invariant Riemannian metric on $T_R\SO3$ is obtained by \eqref{eqn:metric_G}, i.e., for $V,W\in T_R\SO3$, 
\begin{align*}
    \g(V,W) = \pair{R^T V, R^TW}_\liealgebra = \frac12 \mathrm{tr}[V^T W],
\end{align*}
which is the Euclidean metric scaled by the factor $\frac12$. 
Thus, \eqref{eqn:metric_compatible} holds, and the metric is also denoted by $\g(\cdot, \cdot) = \pair{\cdot, \cdot}$. 
As shown in \eqref{eqn:frame_G}, an orthonormal frame on $T_R\SO3$ can be chosen as
\begin{align*}
    \{ E_1, E_2, E_3 \} = \{R\hat e_1, R\hat e_2, R\hat e_3\}.
\end{align*}

Next, we define an orthogonal projection from $\Re^{3\times 3}$ to $T_R\SO3$ as follows.
For $A\in\Re^{3\times 3}$, we have
\begin{align}
    \pair{ A, E_i} & = \frac{1}{2}\mathrm{tr}[ A^T R \hat e_i] = \frac{1}{2}\mathrm{tr}[\mathrm{skew}(R^T A)^T \hat e_i]\nonumber \\
                   & = \pair{ \mathrm{skew}(R^TA), \hat e_i}.
\end{align}
where $\mathrm{skew}:\Re^{3\times 3}\rightarrow\so3$ is defined as $\mathrm{skew}(B) = \frac{1}{2}(B-B^T)$ for $B\in\Re^{3\times 3}$.
The second equality is from the fact that the trace of the product of a symmetric matrix and a skew-symmetric matrix is always zero, i.e., only the skew-symmetric component of $R^TA$ contributes to the inner product, as $\hat e_i$ is skew-symmetric. 

Then, from \eqref{eqn:P}, the projection operator $P:\Re^{3\times 3}\rightarrow T_R\SO3$ applied to $A\in\Re^{3\times 3}$ is
\begin{align*}
    (P(R))(A)& = \sum_{i=1}^3 \pair{A, E_i} E_i\\
             &= R \parenth{ \sum_{i=1}^3  \pair{\mathrm{skew}(R^TA), \hat e_i}\hat e_i} \\
             & = R\, \mathrm{skew} (R^TA),
\end{align*}
where the last equality is from the fact that $\{\hat e_1, \hat e_2, \hat e_3\}$ is an orthonormal basis of $\so3$. 
This can be further rearranged into
\begin{align}
    (P(R))(A) = \frac{1}{2} (A - RA^T R), \label{eqn:P_SO3}
\end{align}

\paragraph{Mean Curvature Vector}
Next, we find the mean curvature vector from \eqref{eqn:H_G_unimodular}.
Let $X,Y\in\mathfrak{X}(\SO3)$ be left-invariant vector fields obtained by $X=R\eta_x$ and $Y=R\eta_y$ for $\eta_x,\eta_y\in\so3$.
Since the covariant derivative in $\Re^{3\times 3}$ can be interpreted as the usual directional derivative,
\begin{align}
    \bar \nabla_{\bar X} \bar Y & = \frac{d}{d\epsilon}\bigg|_{\epsilon =0} Y (R + \epsilon R\eta_x) = \frac{d}{d\epsilon}\bigg|_{\epsilon =0} (R + \epsilon R\eta_x)\eta_y\nonumber \\
                                & = R\eta_x \eta_y. \label{eqn:cov_deriv_proj_SO3}
\end{align}
Therefore, from \eqref{eqn:H_G_unimodular},
\begin{align*}
    H = \sum_{i=1}^3 \bar\nabla_{\bar E_i} \bar E_i = R \sum_{i=1}^3 \hat e_i^2 = -2R,
\end{align*}
where we have used the fact that $\sum_{i=1}^3 \hat e_i^2 = -2 I_{3\times 3}$.

\paragraph{SDE for Brownian Motion}

According to \Cref{cor:Brownian_motion_G_unimodular_Emb}, the stochastic differential equations for the Brownian motion on $\SO3$ embedded in $\Re^{3\times 3}$ are given by
\begin{gather}
    dR  = \sum_{i=1}^3  R\hat e_i \circ dW_i, \label{eqn:SDE_Brownian_SO3}\\
    dR  = - R \, dt + \sum_{i=1}^3 R\hat e_i dW_i. \label{eqn:SDE_Brownian_SO3_Ito}
\end{gather}
We can show $d(R^T R) = 0$ to ensure that the sample trajectory of the above stochastic differential equation evolves on $\SO3$. 

Since $\SO3$ is a unimodular Lie group, the drift term of the Stratonovich term vanishes as in \eqref{eqn:SDE_Brownian_SO3}.
But, as it is considered as an embedded manifold, the It\^{o} formulation of \eqref{eqn:SDE_Brownian_SO3_Ito} has a drift given by the mean curvature vector. 

The drift term obtained by the mean curvature vector,  namely $-Rdt$ has been referred to as the \textit{pinning} drift in~\cite{marjanovic2018numerical,piggott2016geometric}.
In this paper, we did not include the pinning drift to enforce the constraint $d(R^T R)=0$. 
Instead, it is naturally obtained by making sure that the generator coincides with the intrinsic Laplace--Beltrami operator on $\SO3$. 

\paragraph{Monte-Carlo Simulation}

We verify a property of the Brownian motion on $\SO3$ numerically as follows.
Let $f(R) = R_{ij} = e_i^T R e_j$.
We check the evolution of its mean using the generator.
The Hessian of $f(R)$ vanishes in the ambient $\Re^{3\times 3}$.
Therefore, from \eqref{eqn:laplace_Ei_Hess} and \eqref{eqn:Hess_proj},
\begin{align*}
    \mathcal{A}f = \frac{1}{2}\Delta f(R) & = \frac{1}{2}\sum_{i=1}^3 (\mathrm{II}(E_i, E_i))[f].
\end{align*}
Using the definition of the mean curvature vector \eqref{eqn:mean_curvature}, this reduces to
\begin{align*}
    \frac{1}{2}\Delta f(R) & = \frac{1}{2} H[f] = -R[f] = -R_{ij}.
\end{align*}
Given that $\frac{1}{2} \Delta R_{ij} = - R_{ij}$ is satisfied element-wise, this implies that $\frac{1}{2}\Delta R = -R$.

Let $M_t = \E[R_t| R_0 = R_{init}] \in\Re^{3\times 3}$ be the mean of the rotation matrix following \eqref{eqn:SDE_Brownian_SO3}, initialized with $R_0 = R_{init}\in\SO3$.
From Dynkin's formula~\cite[Section 7.4]{oksendal2003stochastic},
\begin{align}
    \dot M_t = \E[\mathcal{A} R_t] = \E[\frac{1}{2}\Delta R_t] = \E[-R_t] = - M_t. \label{eqn:M_dot}
\end{align}
Therefore, $M_t = e^{-t} M_0$ with $M_0=R_{init}$, indicating that each element of the mean of the rotation matrix decays exponentially.
This can be interpreted as the resulting probability density converging exponentially to the uniform distribution on $\SO3$.

Furthermore, the solution of \eqref{eqn:M_dot} can be utilized to show that its Frobenius norm $\|M_t\|_F = \sqrt{\mathrm{tr}(M_t^T M_t)}$ evolves according to
\begin{align}
    \| M_t \|_F = e^{-t} \|M_0\|_F = \sqrt{3} e^{-t}. \label{eqn:log_Mt}
\end{align}

Now, we verify \eqref{eqn:log_Mt} with a Monte-Carlo simulation.
We generate $N=10,000$ sample trajectories of \eqref{eqn:SDE_Brownian_SO3}, each initialized with $R_0 = I_{3\times 3}$, to approximate $M_t$ by $M_t \approx\frac1N \sum_{i=1}^N R^i_t$.
Then, the logarithm of its Frobenius norm is computed over time and compared with the theoretical prediction from \eqref{eqn:log_Mt}.
The results shown in \Cref{fig:Brownian_SO3} exhibit excellent agreement.\footnote{The simulation code was initially drafted by Google Gemini and subsequently verified and revised by the first author.}

\begin{figure}
    \centerline{\includegraphics[width=0.85\columnwidth]{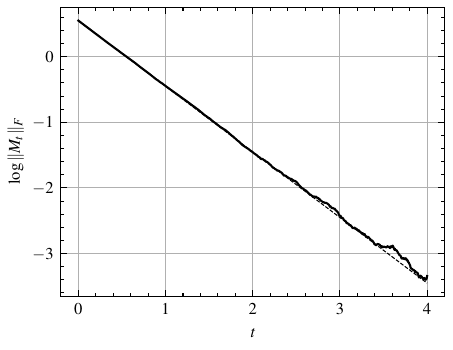}}
    \caption{Numerical verification of the Brownian motion on $\SO3$: the Frobenius norm of the mean of rotation matrices (solid) sampled by \eqref{eqn:SDE_Brownian_SO3} is compared against the theoretical prediction (dashed) from \eqref{eqn:log_Mt}.}
    \label{fig:Brownian_SO3}
\end{figure}

\subsection{Affine Group}\label{sec:Affine}

We formulate Brownian motion on the affine group of $\Re$ given by
\[
    \Aff(\Re)=\left\{\begin{bmatrix}a&b\\0&1\end{bmatrix}:a>0,\,b\in\Re\right\}\subset\Re^{2\times2},
\]
where we consider a component for orientation-preserving transformation, i.e., $a>0$. 
While $\Aff(\Re)$ is embedded in $\Re^{2\times 2}$, the condition \eqref{eqn:metric_compatible} does not hold.
Also it is not unimodular.
Therefore, we use the results of \Cref{thm:Brownian_motion_G}.

\paragraph{Riemannian Metric}

The Lie algebra is
\[
    \mathfrak{aff}(\Re)=\left\{\begin{bmatrix}u & v\\0&0\end{bmatrix}:u,v\in\Re\right\},
\]
for which the inner product is chosen as
\begin{align}
    \pair{\eta, \xi}_\liealgebra = \frac{1}{2}\mathrm{tr}[\eta^T \xi]. \label{eqn:metric_Aff}
\end{align}
where $\eta,\xi \in \mathfrak{aff}(\Re)$.
Specifically, if $\eta=\begin{bmatrix} u & v \\ 0 & 0\end{bmatrix}$ and $\xi=\begin{bmatrix} w & z \\ 0 & 0 \end{bmatrix}$ for $u,v,w,z\in\Re$, we have
$\langle \eta, \xi\rangle_\liealgebra = \frac{1}{2} (uw + vz)$. 
An orthonormal basis of $\mathfrak{aff}(\Re)$ with respect to \eqref{eqn:metric_Aff} is chosen as
\[
    e_1^\liealgebra = \begin{bmatrix}\sqrt{2}&0\\0&0\end{bmatrix}, \qquad
    e_2^\liealgebra = \begin{bmatrix}0&\sqrt{2}\\0&0\end{bmatrix}.
\]

The left-invariant metric on $T_g\Aff(\Re)$ is given by \eqref{eqn:metric_G}, i.e., 
\begin{align*}
    \g(A, B) = \pair{g^{-1}A, g^{-1} B}_{\liealgebra} = \frac{1}{2}\mathrm{tr}[A^T g^{-T} g^{-1} B],
\end{align*}
for $A, B \in T_g \Aff(\Re)$. 
Thus, an orthonormal basis of $T_g\Aff(\Re)$ is chosen as
\begin{align}
    \{E_1, E_2\} = \{g e_1^\liealgebra, g e_2^\liealgebra\} = \{ a e_1^\liealgebra, a e_2^\liealgebra\},
\end{align}
which satisfies $\g(E_i, E_j)=\delta_{ij}$. 

\paragraph{Drift Term}
Now, we find the drift term of the Stratonovich form \eqref{eqn:SDE_Brownian_G}, given by the half of $J$ introduced in \eqref{eqn:J}.

For $\eta=\begin{bmatrix} u & v \\ 0 & 0 \end{bmatrix}\in\mathfrak{aff}(\Re)$, we have
\begin{align*}
    \ad_\eta e_1^\liealgebra & = \eta e_1^\liealgebra - e_1^\liealgebra \eta = - v e_2^\liealgebra, \\
    \ad_\eta e_2^\liealgebra & = \eta e_2^\liealgebra - e_2^\liealgebra \eta = u e_2^\liealgebra,
\end{align*}
such that $\mathrm{tr}[\ad_\eta] = \sum_{i=1}^2 \langle e_i^\liealgebra, \ad_\eta e_i^\liealgebra \rangle_\liealgebra = u $. 
From \eqref{eqn:Jeta}, we have
\begin{align*}
    \pair{ J, \eta}_\liealgebra = -\mathrm{tr}[\ad_\eta] = -u.
\end{align*}
Substituting $\eta = \frac{1}{\sqrt{2}}(u e_1^\liealgebra + v e_2^\liealgebra)$ into the above, we obtain
\begin{align*}
    J = -\sqrt2 e_1^\liealgebra. 
\end{align*}

\paragraph{SDE for Brownian Motion}

According to \Cref{thm:Brownian_motion_G}, the stochastic differential equations for the Brownian motion on $\Aff(\Re)$ embedded in $\Re^{2\times 2}$ are given by
\begin{gather}
    g^{-1} dg  = -\frac{\sqrt2}{2} e_1^\liealgebra\, dt + \sum_{i=1}^2 e_i^\liealgebra \circ dW_i. \label{eqn:SDE_Brownian_Aff}
\end{gather}
Since $\Aff(\Re)$ is not a unimodular group, the Stratonovich formulation has a drift term. 
The It\^{o} formulation can be similarly obtained by \eqref{eqn:SDE_Brownian_Ito_G}.

\section{Conclusions}

This paper presented a unified geometric framework for formulating Brownian motion on manifolds, including intrinsic Riemannian manifolds, embedded submanifolds, and Lie groups.  
The construction was based on the principle that Brownian motion is the stochastic process whose generator equals one-half of the Laplace--Beltrami operator.  
By injecting isotropic noise along orthonormal frames and deriving the necessary drift term to satisfy this generator condition, both Stratonovich and It\^{o} stochastic differential equations were obtained in explicit geometric form.  

The results reveal that the drift term admits a clear geometric interpretation: it corresponds to the covariant derivatives of the frame fields for intrinsic manifolds, the mean curvature vector for embedded manifolds, and the adjoint-trace term for Lie groups, which vanishes for unimodular cases.  
These findings unify previously disparate formulations of Brownian motion under a single geometric interpretation, bridging intrinsic, extrinsic, and algebraic perspectives.  

Beyond theoretical clarity, the framework provides a consistent foundation for stochastic modeling and numerical simulation on nonlinear configuration spaces arising in mechanics, control, and robotics.  
Future work may extend this approach to more general diffusion processes, manifolds with boundary or non-metric connections, and data-driven stochastic modeling on manifolds.

Finally, the outcomes of this paper can be summarized as a systematic method for constructing Brownian motion on an arbitrary Riemannian manifold:
\begin{itemize}
    \item \textit{Diffusion:} The diffusion term is obtained by injecting noise along each axis of an orthonormal frame.
        Depending on the type of manifold, an orthonormal frame can be chosen as follows:
        \begin{itemize}
            \item \emph{Riemannian manifold:} The frame can be chosen locally, often specified in coordinates.
            \item \emph{Embedded Riemannian manifold:} The ambient Euclidean frame can be projected onto the tangent space of the manifold to obtain the pseudo-frame, as presented in \Cref{def:pseudo-frame}.  
                This formulation is globally defined, though the number of diffusion terms exceeds the manifold dimension.
            \item \emph{Lie group:} A basis of the Lie algebra can be left-translated to construct a global orthonormal frame of the tangent bundle, as shown in \eqref{eqn:frame_G}.            
        \end{itemize}
    \item \textit{Drift in the Stratonovich formulation:} 
        The drift term is chosen so that the resulting infinitesimal generator coincides with one-half of the Laplace--Beltrami operator.
        For the Stratonovich formulation, the drift term is given as follows:
        \begin{itemize}
            \item \emph{Riemannian manifold:} Negative one-half of the sum of the covariant derivatives of each frame vector along itself. 
            \item \emph{Embedded manifold:} Negative one-half of the sum of the covariant derivatives of each pseudo-frame vector along itself. 
            \item \emph{Lie group:} One-half of the sum of the co-adjoint operators of each Lie algebra basis element along itself.  
                The drift term vanishes when the group is unimodular.
        \end{itemize}
    \item \textit{Drift in the It\^{o} formulation:}
        \begin{itemize}
            \item \emph{Riemannian manifold or Lie group:} The drift term vanishes. 
            \item \emph{Embedded Riemannian manifold or embedded Lie group satisfying \eqref{eqn:metric_compatible}:} The drift term equals one-half of the mean curvature vector.  
                For any unimodular group, the mean curvature vector simplifies as in \eqref{eqn:H_G_unimodular}.
        \end{itemize}
\end{itemize}

\bibliographystyle{ieeetran}
\bibliography{ref}

\end{document}